\documentclass[11pt,final]{amsart}

\usepackage{amssymb,amsmath,latexsym}
\usepackage{amsthm}
\usepackage{amsfonts}
\usepackage{enumerate}
\usepackage{booktabs}
\usepackage{mathrsfs}

\usepackage{manfnt}

\usepackage{hyperref}
\usepackage{xcolor}
\hypersetup{
    colorlinks, linkcolor={red!80!black},
    citecolor={blue!80!black}, urlcolor={blue}
}
\usepackage[color]{showkeys}
\definecolor{refkey}{gray}{.75}
\definecolor{labelkey}{gray}{.75}

\usepackage{graphicx,pgf}
\usepackage{float}

\usepackage{caption}
\usepackage{subcaption}

\usepackage{fancyhdr}

\newtheorem{theorem}{Theorem}[section]

\newtheorem{proposition}{Proposition}[section]
\newtheorem{lemma}{Lemma}[section]
\newtheorem{remark}{Remark}[section]

\newcommand{\aiminabs}[1]{\lvert #1 \rvert}
\newcommand{\aiminnorm}[1]{\| #1 \|}
\newcommand{\aimininner}[2]{\langle #1, #2 \rangle}

\numberwithin{equation}{section}
\numberwithin{figure}{section}

\let\oldtocsection=\tocsection
\let\oldtocsubsection=\tocsubsection
\let\oldtocsubsubsection=\tocsubsubsection
\renewcommand{\tocsection}[2]{\hspace{0em}\oldtocsection{#1}{#2}}
\renewcommand{\tocsubsection}[2]{\hspace{2em}\oldtocsubsection{#1}{#2}}
\renewcommand{\tocsubsubsection}[2]{\hspace{4em}\oldtocsubsubsection{#1}{#2}}

\begin{document}
\title{The global attractor of the 2d Boussinesq equations with fractional Laplacian in subcritical case}
\author{Aimin Huang}
\author{Wenru Huo}
\address{The Institute for Scientific Computing and Applied Mathematics, Indiana University, 831 East Third Street, Rawles Hall, Bloomington, Indiana 47405, U.S.A.}
\address{The Department of Mathematics, Indiana University, 831 East Third Street, Rawles Hall, Bloomington, Indiana 47405, U.S.A.}
\email{AH:huangepn@gmail.com}
\email{WH:whuo@imail.iu.edu}


\date{\today}

\begin{abstract}
We prove global well-posedness of strong solutions and existence of the global attractor for the 2D Boussinesq system in a periodic channel with fractional Laplacian in subcritical case. The analysis reveals a relation between the Laplacian exponent and the regularity of the spaces of velocity and temperature.
\end{abstract}

\maketitle

\setcounter{tocdepth}{2}
\addtocontents{toc}{~\hfill\textbf{Page}\par}

\section{Introduction}
This paper studies the existence of global attractor for the solution operator ${S(t)}$ to the two-dimensional (2D) incompressible Boussinesq equations with subcritical dissipation. The 2D Boussinesq equations read
\begin{equation}\begin{cases}\label{eq1.1.1}
\partial_t\boldsymbol u + \boldsymbol u\cdot \nabla \boldsymbol u + \nu(-\Delta)^\alpha \boldsymbol u=- \nabla \pi + \theta \boldsymbol e_2, \qquad x \in \Omega, \hspace{2pt} t > 0, \\
\nabla \cdot \boldsymbol u=0,  \qquad\qquad\qquad\qquad\qquad\qquad\qquad x \in \Omega, \hspace{2pt} t > 0,\\
\partial_t \theta +\boldsymbol u\cdot \nabla \theta+ \kappa(-\Delta)^{\beta}\theta = f, \qquad\qquad\qquad x \in \Omega, \hspace{2pt} t > 0,\\
\boldsymbol u(x,0)=\boldsymbol u_0(x),\quad \theta(x,0)=\theta_0(x), \qquad \qquad x \in \Omega, \hspace{2pt} t > 0,
\end{cases}\end{equation}
where $\Omega=[0, L]^2$ is the periodic domain, $\nu>0$ the fluid viscosity, and $\kappa>0$ the  diffusivity;  $\boldsymbol u =\boldsymbol u(x,t)=(u_1(x,t), u_2(x,t))$ denotes the velocity, $\pi=\pi(x,t)$  the pressure, $\theta= \theta(x,t)$ a scalar function  which may for instance represents the temperature variation in the content of thermal convection, $\boldsymbol e_2=(0,1)$  the unit vector in the vertical direction, and $f=f(x)$  a time-independent forcing term.
Since in this article we consider 2D Boussinesq equations with a subcritical dissipation, we assume that the exponents $\alpha$ and $\beta$ satisfy
\begin{equation} \label{eq1.2}
\alpha,\; \beta \in (\frac{1}{2},1).
\end{equation}
Recently, the 2D Boussinesq equations and their fractional generalizations have attracted considerable attention due to their physical applications and mathematical significance. When $\alpha=\beta=1$, the system \eqref{eq1.1.1} is then called the standard 2D Boussinesq equations, which  are widely used to model the geophysical flows such as atmospheric fronts and oceanic circulation and also play an important role in the study of Rayleigh-B\'enard convection (c.f. \cite{Ped87}).
Besides, the 2D Boussinesq system  are the two-dimensional models which retain the key vortex-stretching mechanism as the 3D Navier-Stokes/Euler equations for axisymmetric swirling flows (c.f. \cite{MB02}).

In the mathematical respect, the global well-posedness, global regularity of the standard 2D Boussinesq system as well as the existence of the global attractor have been widely studied, see for example \cite{CD80, FMT87, ES94, CN97, MZ97, Wang05, CLR06, Wang07, KTW11}.
Recently, there are many works devoted to the study of the 2D Boussinesq system with partial viscosity, see for example \cite{HL05, Cha06, HK07, DP09, HK09,  HKR11} in the whole space $\mathbb R^2$ and \cite{Zha10, LPZ11, HKZ13, Hua15b} in bounded domains. There are also many works which considered the global regularity of 2D Boussinesq system with fractional diffusion, see for example \cite{WX12, JMWZ14, XX14, YJW14, SW14}.

In some realistic applications, the variation of the  viscosity and diffusivity with the temperature may not be disregarded (see for example \cite{LB96} and references therein) and there are many works on this direction, too, see for example \cite{LB96, LB99, SZ13, LPZ13, Hua15a} where the existence of weak solutions, global regularity, and existence of global attractor have been studied.
However, the global regularity for the inviscid 2D Boussinesq system where $\nu=\kappa=0$ is still an outstanding open problem and the study of fractional Boussinesq system may shed some light on the inviscid Boussinesq system.
To the best of our knowledge, the existence of the global attractor for the 2D Boussinesq equations with fractional dissipation has not been addressed yet, which is the goal of this article.

This work is motivated by the \cite{CC04, Ju05}, where the existence of global attractor of the 2D subcritical SQG equations has been proved. The key point in \cite{CC04, Ju05} is that they proved a positivity lemma for the fractional Laplacian. Armed with the positivity lemma, we can similarly obtain the maximum principle for $\theta$ as in \cite{CC04, Ju05} and then follow the standard procedure to show the existence of global attractor. We point out that our estimates are much more involved and harder than those in \cite{Ju05} since the velocity is given by an evolution equation instead of the Riesz potential of $\theta$ and we have to use a variant of uniform Gronwall lemma (see Lemma \ref{lemma 2.2}) to prove the uniform bound for the velocity $\boldsymbol u$.

In this article, we prove the existence of a global attractor for \eqref{eq1.1.1} in the space $H^{s_1} \times H^{s_2}$, where
\begin{equation}\label{cond1}
s_1> 2\mathrm{max}\{1-\alpha, 1-\beta \}, \qquad s_2 \geq 1,
\end{equation}
and
\begin{equation} \label{cond2}
0 \leq s_2-s_1 < \alpha+\beta.
\end{equation}
Condition\eqref{cond1} is natural for $s_1$ in that the smaller $\alpha$ and $\beta$ are, the less regularity we can obtain from the equation \eqref{eq1.1.1}$_3$, and hence we need to assume more regularity for the initial data $\theta_0$ (that is larger $s_1$) to compensate for reduced smoothing effect. We need at least $H^1$-regularity (that is $s_2\geq 1$) for $\boldsymbol u$ to show the $L^p$-estimate for $\boldsymbol u$, see Section \ref{sec3.3}. An interpretation of condition \eqref{cond2} is that the interplay of $\theta$ and $\boldsymbol u$ in \eqref{eq1.1.1}  restricts the gap between $s_1$ and $s_2$.

The roadmap of this article is as follows. In Section \ref{sec2}, we introduce the notation, some preliminary results, state our main results, invoke the results from \cite{Tem88} to prove the existence of the global attractor in $H^{s_1} \times H^{s_2}$ space for $s_1$ and $s_2$ satisfying \eqref{cond1} and \eqref{cond2}.  Section \ref{sec3} is devoted to the proof of the $H^s$-uniform bounds and for $(\theta, \boldsymbol u)$, where we first show the uniform $L^p$-estimates for $p \geq 2$ in subsections \ref{sec3.1}-\ref{sec3.3}. Then, in subsection \ref{sec3.4}, we prove that the compactness of the absorbing ball in $H^{s_1} \times H^{s_2}$, which is the key of proof of the compact global attractor.
In Section \ref{sec4}, we show that the solution operator $\{S(t): t \geq 0 \}$ is continuous from $\mathbb{R}_{+}$ to
$H^{s_1} \times H^{s_2}$ and it is continuous from $H^{s_1} \times H^{s_2}$ to $H^{s_1} \times H^{s_2}$ for fixed $t \geq 0$. These also give a proof of the existence and the uniqueness of the strong solution for \eqref{eq1.1.1}.
\section{Notations and preliminaries} \label{sec2}
\subsection{Notations and function spaces}
We denote by $\mathcal{C}(I, X)$ the space of all continuous functions from the interval $I$ to some normed space $X$, and by $L^{p}(\Omega)$ the space of the $p$th-power integrable functions normed by
$$\aiminnorm{f}_{L^p}=\left(\int_{\Omega} |f(x)|^p {\mathrm{d}x}\right)^{\frac{1}{p}},
\qquad \aiminnorm{f}_{L^{\infty}}= \mathrm{ess}\sup\limits_{x \in \Omega} |f(x)|.$$
We also denote by $L^{p}(0,T;X)$ the space of all measurable functions
$u : [0,T] \mapsto X$ with the norm
\[
\aiminnorm{u}^p_{L^{p}(0,T;X)}= \int_0^T \aiminnorm{u}^p_{X} {\mathrm{d}t},
\qquad \aiminnorm{u}_{L^{\infty}(0,T;X)}=  \mathrm{ess}\sup\limits_{t \in [0,T]}
\aiminnorm{u}_{X}.
\]
The Fourier transform $\hat{f}$ of a tempered distribution $f(x)$ on the periodic domain $\Omega$ is defined as
$$\widehat{f}(k)= \frac{1}{(2\pi)^2} \int_{\Omega} f(x)e^{-ik \cdot x} {\mathrm{d}x}, $$
where $k=(k_1,k_2)$ is a tuple consisting two integers.
We denote the square root of the Laplacian $(-\Delta)^{\frac{1}{2}}$ by $\Lambda$ and we have
$$\widehat{\Lambda f}(k)= |k| \widehat{f}(k),$$
where $\aiminabs{k}=\sqrt{k_1^2+k_2^2}$. We define the fractional Laplacian $\Lambda^{s}f$ for
$s \in \mathbb{R}$ by its Fourier series
\[
\Lambda^{s}f := \sum_{k \in \mathbb{Z}^2} \aiminabs{k}^{s} \widehat{f}(k)e^{ik \cdot x}.
\]
For any tempered distribution $f$ on $\Omega$ and $s \in \mathbb{R}$,  we define the norm
$$\aiminnorm{f}_{H^s}= \aiminnorm{\Lambda^s f}_{L^2}=\left (\sum_{k \in \mathbb{Z}^2}
|k|^{2s} |\widehat{f}(k)|^2 \right)^{\frac{1}{2}}$$
and $H^{s}(\Omega)$ denotes the space of the functions $f$ such that
$\aiminnorm{f}_{H^s}$ is finite. For $1 \leq p \leq \infty$ and $s \in \mathbb{R}$, the space $H^{s,p}(\Omega)$ consists of the functions $f$
such that $f= \Lambda^{-s} g$ for some $g \in L^p(\Omega)$. The $H^{s,p}$-norm of $f$ is defined by
\[
\aiminnorm{f}_{H^{s,p}} = \aiminnorm{\Lambda^{s} f}_{L^{p}}.
\]
\subsection{Some preliminary results}
We first recall the Gagliardo-Nirenberg interpolation inequality, which is used frequently in this article.
\begin{lemma}[The interpolation inequality]\label{lemma2.0}
For any $s_1\leq s\leq s_2$, we have
\[
\aiminnorm{ \Lambda^s g}_{L^2} \leq \aiminnorm{ \Lambda^{s_1} g}_{L^2}^{\delta} \aiminnorm{ \Lambda^{s_2} g}_{L^2}^{1-\delta},
\]
where $s= \delta s_1 +  (1-\delta) s_2$ for some $0 \leq \delta \leq 1$.
\end{lemma}
Next, we recall the Gronwall Lemma and Uniform Gronwall Lemma.
\begin{lemma}[Gronwall Lemma] \label{lemma2.0.1}
Let $g$, $h$, $y$ and $\frac{{\mathrm{d}y}}{{\mathrm{d}t}}$ be locally integrable functions on $(t_0, +\infty)$ such that
\[
\frac{{\mathrm{d}y}(t)}{{\mathrm{d}t}} \leq g(t)y(t) + h(t), \qquad \forall t \geq t_0,
\]
then $y(t)$ satisfies the following inequality,
\[
y(t) \leq y(t_0)\mathrm{exp}\left(\int_{t_0}^t g(s) {\mathrm{d}s} \right)+
\int_{t_0}^t h(s) \mathrm{exp}\left(\int_{s}^t g(\tau) {\mathrm{d}\tau} \right){\mathrm{d}s}.
\]
\end{lemma}
\begin{lemma}[Uniform Gronwall Lemma]\label{lemma 2.1}
Let $g$, $h$ and $y$ be non-negative locally integrable functions on $(t_0, +\infty)$ such that
$$\frac{{\mathrm{d}y}(t)}{{\mathrm{d}t}} \leq g(t)y(t)+h(t), \qquad \forall t \geq t_0,$$
and
$$\int_{t}^{t+r} g(s) {\mathrm{d}s} \leq a_1, \qquad \int_{t}^{t+r} h(s) {\mathrm{d}s} \leq a_2, \qquad \int_{t}^{t+r} y(s) {\mathrm{d}s} \leq a_3, \qquad  \forall t \geq t_0,$$
where $r, a_1, a_2$ and $a_3$ are positive constants. Then
$$y(t+r) \leq \left(\frac{a_3}{r}+a_2\right)e^{a_1}, \qquad  \forall t \geq t_0.$$
\end{lemma}
For the proof of uniform Gronwall Lemma, one can refer to \cite[pp. 91] {Tem88} and \cite{FP67} for its proof. We now present a variant of uniform Gronwall lemma, which is the key to prove the uniform estimates for the velocity $\boldsymbol u$. We refer the interesting readers to \cite{GPZ09, Pat11} for other variants of the uniform Gronwall lemma.
\begin{lemma}\label{lemma 2.2}
Let $\lambda>0$ be a positive constant and $g(t)$, $y(t)$ be non-negative locally integrable functions on $(0, +\infty)$ such that
\[
\frac{{\mathrm{d}y}(t)}{{\mathrm{d}t}} + \lambda y(t) \leq g(t), \qquad \forall t \geq t_0,
\]
and
\[\int_{t}^{t+1} g(s) {\mathrm{d}s} \leq a_1, \qquad  \forall t \geq t_0>0.
\]
Then, there exists $t^{*}(t_0, y_0) >0$ large enough, such that for all $t > t^{*}(t_0, y_0)$, we have
\[
y(t)\leq C,
\]
for some constant $C>0$ independent of the time $t$ and initial data $y_0$.
\end{lemma}
\begin{proof} Multiplying $e^{\lambda t}$ and integrating on both side of the equation above from $t_0$ to some $t > t_0$, we have
$$y(t) \leq e^{-\lambda (t-t_0)}y(t_0)+ e^{-\lambda t} \int_{t_0}^t e^{\lambda s} g(s) {\mathrm{d}s}.$$
Suppose $t_0+m \leq t \leq t_0+m+1$ for some integer $m$, then
\begin{equation} \label{eq2.2.1}
\begin{split}
y(t)
&\leq e^{-\lambda (t-t_0)}y(t_0) + e^{-\lambda t} \sum_{k=0}^{m} \int_{t_0+k}^{t_0+k+1} e^{\lambda s} g(s) {\mathrm{d}s} \\
&\leq e^{-\lambda (t-t_0)}y(t_0) + e^{-\lambda t} \sum_{k=0}^{m}  e^{\lambda (t_0+k+1)} \int_{t_0+k}^{t_0+k+1} g(s) {\mathrm{d}s} \\
&\leq e^{-\lambda (t-t_0)}y(t_0) + a_1 e^{-\lambda t} \sum_{k=0}^{m}  e^{\lambda (t_0+k+1)} \\
&\leq e^{-\lambda (t-t_0)}y(t_0)+ a_1 e^{\lambda (t_0+m+1-t)} \frac{e^{\lambda}}{e^{\lambda}-1}  \\
&\leq e^{-\lambda (t-t_0)}y(t_0)+ a_1\frac{e^{2\lambda}}{e^{\lambda}-1}.
\end{split}
\end{equation}
We thus finished the proof of Lemma~\ref{lemma 2.2}.
\end{proof}
We now recall an improved positivity lemma due to \cite{Ju05}.
\begin{lemma}[Improved Positivity Lemma]\label{lemma 2.3}
Suppose $s \in [0,2]$ and $\theta$, $\Lambda^s \theta \in L^p(\Omega)$. Then
$$\int_{\Omega} |\theta|^{p-2}\theta \Lambda^s \theta {\mathrm{d}x} \geq
\frac{2}{p} \int_{\Omega} \left(\Lambda^{\frac{s}{2}} |\theta|^{\frac{p}{2}}\right)^2 {\mathrm{d}x}.$$
\end{lemma}
We will use the following Kate-Ponce and commutator inequalities from \cite{KP88}, see also \cite{WU02, Ju05}.
\begin{lemma} \label{lem2.4}
Suppose that $f, g \in C^{\infty}_c(\Omega)$,  then
\begin{equation} \label{eqn2}
\aiminnorm{\Lambda^{s}(fg)}_{L^r} \leq
C(\aiminnorm{\Lambda^{s}f}_{L^{p_1}}\aiminnorm{g}_{L^{p_2}}+
\aiminnorm{\Lambda^{s}g}_{L^{q_1}}\aiminnorm{f}_{L^{q_2}})
\end{equation}
where $s >0$, $1 < r \leq p_1, p_2, q_1, q_2 \leq \infty$ and $1/r=1/p_1+1/{p_2}= 1/{q_1}+1/{q_2}$.
\end{lemma}
\begin{lemma} \label{lem2.5}
Suppose that $f, g \in C^{\infty}_c(\Omega)$,  then
\begin{equation} \label{eqn3}
\aiminnorm{\Lambda^{s}(f \cdot \nabla g)- f \cdot (\Lambda^{s}\nabla g)}_{L^2}
\leq C(\aiminnorm{\nabla f}_{L^{p_1}}\aiminnorm{\Lambda^{s} g}_{L^{p_2}}+\aiminnorm{\Lambda^{s} f}_{L^{q_1}}\aiminnorm{\nabla g}_{L^{q_2}}).
\end{equation}
where $s >0$, $2 < p_1, p_2, q_1, q_2 \leq \infty$ and $1/2=1/p_1+1/{p_2}= 1/{q_1}+1/{q_2}$.
\end{lemma}
\begin{remark}
We remark that the inequalities \eqref{eqn2} and \eqref{eqn3} in Lemmas \ref{lem2.4} and \ref{lem2.5} are also valid for those $f$
and $g$ belonging to certain Sobolev spaces which make the right-hand sides of \eqref{eqn2} and \eqref{eqn3} to be finite.
\end{remark}

\subsection{The main results}
We now state the result about the existence of weak solution and global strong solutions of the 2D Boussinesq system \eqref{eq1.1.1}. The proof involves the standard Galerkin approximation, some basic functional analysis theorems in \cite{Tem84}, the uniform estimates in Section~\ref{sec3}, and continuity estimates in Section~\ref{sec4} .
\begin{theorem} \label{thm2.0}
Let $$H_1 = \left\{\theta \in L^2(\Omega): \int_{\Omega} \theta {\mathrm{d}x} =0 \right\}$$ and $$ H_2 = \left\{\boldsymbol u \in L^2(\Omega)^2: \nabla \cdot \boldsymbol u =0,
\int_{\Omega} u_1 {\mathrm{d}x} =\int_{\Omega} u_2 {\mathrm{d}x}=0 \right\}.$$ Suppose $f \in H^{-\beta}$ and $(\theta_0, \boldsymbol u_0) \in H_1 \times H_2$. Then, for any $T > 0$, there exists at least one weak solution of the 2D Boussinesq equations \eqref{eq1.1.1} in the following sense:
\[
\frac{\mathrm{d}}{\mathrm{d}t} \int_{\Omega} \boldsymbol u \boldsymbol \varphi {\mathrm{d}x} - \int_{\Omega} \boldsymbol u(\boldsymbol u \cdot \nabla \boldsymbol \varphi) {\mathrm{d}x} + \nu \int_{\Omega} (\Lambda^{\alpha} \boldsymbol u)(\Lambda^{\alpha} \boldsymbol \varphi) {\mathrm{d}x} = \int_{\Omega} \theta \boldsymbol e_2 \boldsymbol u {\mathrm{d}x}, \qquad \forall \boldsymbol \varphi=(\varphi_1, \varphi_2) \in C^{\infty} (\Omega)^2,
\]
and
\[
\frac{\mathrm{d}}{\mathrm{d}t} \int_{\Omega} \theta \psi {\mathrm{d}x} - \int_{\Omega} \theta(\boldsymbol u \cdot \nabla \psi) {\mathrm{d}x} + \kappa \int_{\Omega} (\Lambda^{\beta} \theta)(\Lambda^{\beta}\psi) {\mathrm{d}x} = \int f \psi {\mathrm{d}x}, \qquad \forall \psi \in C^{\infty} (\Omega).
\]
Moreover, $\theta \in L^{\infty}(0,T; H_1) \cap L^{2}(0,T; H^{\beta})$ and $\boldsymbol u \in L^{\infty}(0,T; H_2) \cap L^{2}(0,T; H^{\alpha})$.

Furthermore, if we assume that $s_1$, $s_2$ satisfy \eqref{cond1} and \eqref{cond2}, $(\theta_0, \boldsymbol u_0)\in H^{s_1} \times H^{s_2}$ and $f \in H^{s_1-\beta} \cap L^p$ for $p > 2$, then for any $T>0$,
the Boussinesq system \eqref{eq1.1.1} has a unique strong solution $(\boldsymbol u,  \theta)$ satisfying
\begin{equation}\begin{split}
	(\theta, \boldsymbol u) &\in \,\mathcal C([0, T], H^{s_1})\cap \mathcal C([0, T], H^{s_2}),\\
	(\theta_t, \boldsymbol u_t) &\in L^2(0, T; H^{s_1-\beta})\cap L^2(0, T; H^{s_2-\alpha}).\\
\end{split}\end{equation}
\end{theorem}

The main goal here is to prove the existence of the global attractor for the Boussinesq system \eqref{eq1.1.1} and we have the following theorem.
\begin{theorem} \label{thm2.1}
Assume that $\nu >0$, $\kappa >0$, $s_1$ and $s_2$ satisfy \eqref{cond1} and \eqref{cond2},  and $f \in H^{s_1-\beta} \cap L^p$ for $p \geq 2$. Then the solution operator $\{S(t)\}_{t\geq 0}$ of the 2D Boussinesq system: $S(t)(\theta_0, \boldsymbol u_0)=(\theta(t),\boldsymbol u(t))$ defines a semigroup in the space $H^{s_1} \times H^{s_2}$ for all $t\in\mathbb R_+$. Moreover, the following statements are valid:
	\begin{enumerate}
		\item for any $(\theta_0, \boldsymbol u_0)\in H^{s_1} \times H^{s_2}$, $t\mapsto S(t)(\theta_0, \boldsymbol u_0)$ is a continuous function from $\mathbb R_+$ into $H^{s_1} \times H^{s_2}$;
		\item for any fixed $t>0$, $S(t)$ is a continuous and compact map in $H^{s_1} \times H^{s_2}$;
		\item $\{S(t)\}_{t\geq 0}$ possesses a global attractor $\mathcal A$ in the space $H^{s_1} \times H^{s_2}$. The global attractor $\mathcal A$ is compact and connected in $H^{s_1} \times H^{s_2}$ and is the maximal bounded attractor and the minimal invariant set in $H^{s_1} \times H^{s_2}$ in the sense of the set inclusion relation.		
		\end{enumerate}
\end{theorem}

\subsection{The global attractor}
In order to prove the main result Theorem \ref{thm2.1}, we are going to utilize the abstract result from \cite[Chapter I]{Tem88} about semigroups and the existence of their global attractors. One can also refer to
\cite{Lad91}, \cite{BV92}, \cite{HK06}
for additional development of the theory of global attractor.
\begin{theorem}\label{thm4.2}
Suppose that $X$ is a metric space with metric $d(\cdot,\cdot)$ and the semigroup $\{S(t) \}_{t\geq 0}$ is a family of operators from $X$ into $X$ itself such that:
\begin{enumerate}[$\quad(i)$]
\item for each fixed $t>0$, $S(t)$ is continuous from $X$ into itself;
\item for some $t_0>0$, $S(t_0)$ is compact from $X$ into itself;
\item there exists a subset $B_0$ of $X$ which is bounded, and a subset $U$ of $X$ which is open, such that $B_0\subset U\subset X$, and $B_0$ is absorbing in $U$ for the semigroup, that is for any bounded subset $B\subset U$, there exists $t_0=t_0(B)>0$ such that
\[
S(t)B\subset B_0,\quad\quad \forall\,t>t_0(B).
\]
($B_0$ is also called the absorbing set of the semigroup in $U$).
\end{enumerate}
Then $\mathcal A:=\omega(B_0)$, the $\omega$-limit set of $B_0$, is a compact attractor which attracts all the bounded sets of $U$, that is for any bounded set  $B\subset U$,
\[
	\lim_{t\rightarrow+\infty}\text{dist}(S(t)B,\mathcal A) = 0,
\]
where $\text{dist}(B_1,B_2):=\sup_{x\in B_1}\inf_{y\in B_2}d(x,y)$ is the non-symmetric Hausdorff distance between subsets of $X$. Furthermore, the set $\mathcal A$ is the maximal bounded attractor in $U$ for the inclusion relation and the minimal invariant
set in the sense of $S(t)\mathcal A= \mathcal A$, $\forall t \geq 0$.

Suppose in addition that $X$ is a Banach space, $U$ is convex and
\begin{enumerate}[$\quad(i)$]
\setcounter{enumi}{3}
\item for any $x\in X$, $t\mapsto S(t)x$ from $\mathbb R_+$ to $X$ is continuous.
\end{enumerate}
Then $\mathcal A=\omega(B_0)$ is also connected.

If $U=X$, $\mathcal A$ is the global attractor of the semigroup $\{S(t) \}_{t\geq 0}$ in $X$.
\end{theorem}
\begin{remark}
We will carry out the proof of Theorem \ref{thm2.1} by checking all the items in Theorem \ref{thm4.2}.
We  first show the $H^s$-uniform estimates in Section~\ref{sec3.4}, which implies items $(ii)$, $(iii)$,  then we check item $(iv)$ in Section~\ref{sec4.1} by proving that $\{S(t)\}_{t\geq 0}$ is continuous from $\mathbb R_+$ to $H^{s_1}\times H^{s_2}$, and finally we check item $(i)$  in Section~\ref{sec4.2} by proving that $\{S(t)\}_{t\geq 0}$ is continuous in $H^{s_1}\times H^{s_2}$.
\end{remark}

\section{Uniform estimates} \label{sec3}
In the following, we denote by $C$  a positive constant, which is independent of time $t$ and of the initial data $\boldsymbol u_0$ and $\theta_0$. The constant $C$  may vary from line to line.
\subsection{$L^2$ and $L^p$-estimate for $\theta$} \label{sec3.1}
\begin{proposition}[Existence of absorbing ball in $L^2$ and $L^p$ for $\theta$]\label{prop3.1}
Under the assumptions of Theorem~\ref{thm2.1}, there exists $ t_1^{*} = t_1^{*}(\aiminnorm{\theta_0}_{L^2}) > 0$, such that
\begin{equation} \label{eq3.1.0}
\aiminnorm{\theta(t)}_{L^2} \leq C, \qquad \forall t \geq t_1^{*},
\end{equation}
and
\begin{equation} \label{eq3.1.3}
\int_{t}^{t+1} \aiminnorm{\Lambda^{\beta} \theta}_{L^2}^2 {\mathrm{d}s} \leq C, \qquad \forall t \geq t_1^{*}.
\end{equation}
\end{proposition}
\begin{proof}
Here and throughout this article, we let $\lambda_1$ be the first eigenvalue of $\Lambda$. Then by the results from  \cite[Section 5.1]{Ju05}, we obtain
\begin{equation}\label{eq3.1.1}
\aiminnorm{\theta(t)}_{L^2}^2 \leq e^{- \kappa \lambda_1^{2\beta} t}
\left(\aiminnorm{\theta_0}_{L^2}^2 - \frac{\aiminnorm{f}_{L^2}^2}{\kappa^2 \lambda_1^{4\beta}}\right)+ \frac{\aiminnorm{f}_{L^2}^2}{\kappa^2 \lambda_1^{4\beta}},
\end{equation}
which immediately implies the uniform bound \eqref{eq3.1.0} for some $ t_1^{*} > 0$ large enough.
Furthermore, taking inner product of \eqref{eq1.1.1}$_3$ with $\theta$ in $L^2(\Omega)$ and integrating in time give
\begin{equation}\label{eq3.1.2.2}
\aiminnorm{\theta(t+1)}_{L^2}^2 + \kappa \int_{t}^{t+1}  \aiminnorm{\Lambda^{\beta} \theta}_{L^2}^2 {\mathrm{d}s} \leq \aiminnorm{\theta(t)}_{L^2}^2+ \frac{\aiminnorm{f}_{L^2}^2}{\kappa \lambda_1^{2\beta}}.
\end{equation}
Therefore, the time average estimate \eqref{eq3.1.3} follows from
\eqref{eq3.1.0}.
We also deduce the equation (5.4) in \cite{Ju05} that for all $2 \leq p \leq \infty$, we have
\begin{equation}\label{eq3.1.2}
\aiminnorm{\theta(t)}_{L^p} \leq e^{- \frac{\kappa \lambda_1^{2\beta}}{p} t}
\left(\aiminnorm{\theta_0}_{L^p} - \frac{p \aiminnorm{f}_{L^p}}{\kappa \lambda_1^{2\beta}}\right)+ \frac{p\aiminnorm{f}_{L^p}}{\kappa \lambda_1^{2\beta}}.
\end{equation}
The above inequality gives the uniform $L^p$ estimate and absorbing ball in $L^p$ for $\theta$ whenever $\theta_0 \in L^p(\Omega)$ for all $p \in [2, \infty)$.
\end{proof}
\subsection{$L^2$-estimate for $\boldsymbol u$} \label{sec3.2}
\begin{proposition}[Existence of absorbing ball in $L^2$ for $\boldsymbol u$]\label{prop3.2}
Under the assumptions of Theorem~\ref{thm2.1}, there exists $t_2^{*}=t_2^{*}(\aiminnorm{\theta_0}_{L^2}, \aiminnorm{\boldsymbol u_0}_{L^2})  > 0$, such that
\begin{equation} \label{eq3.1.4}
\aiminnorm{\boldsymbol u(t)}_{L^2} \leq C, \qquad \forall t \geq t_2^{*},\end{equation}
and
\begin{equation} \label{eq3.1.8}
\int_{t}^{t+1} \aiminnorm{\Lambda^{\alpha} \boldsymbol u}_{L^2}^2 {\mathrm{d}s} \leq C, \qquad \forall t \geq t_2^{*}.
\end{equation}
\end{proposition}
\begin{proof} Taking the inner product of the equation \eqref{eq1.1.1}$_1$ with $\boldsymbol u$ in $L^2$, since $\boldsymbol u$ is divergent free and
$\aimininner{\boldsymbol u\cdot \nabla \boldsymbol u}{\boldsymbol u}=0$,  we have
$$\frac{1}{2} \frac{\mathrm{d}}{\mathrm{d}t} \aiminnorm{\boldsymbol u}_{L^2}^2 + \nu \aiminnorm{\Lambda^{\alpha} \boldsymbol u}_{L^2}^2 = \aimininner{\theta e_2}{\boldsymbol u}
\leq \frac{1}{2\nu} \aiminnorm{\Lambda^{-\alpha} \theta}_{L^2}^2 + \frac{\nu}{2}\aiminnorm{\Lambda^{\alpha} \boldsymbol u}_{L^2}^2.$$
Then,
\begin{equation} \label{eq3.1.9}
\frac{\mathrm{d}}{\mathrm{d}t} \aiminnorm{\boldsymbol u}_{L^2}^2 + \nu \aiminnorm{\Lambda^{\alpha} \boldsymbol u}_{L^2}^2
\leq  \frac{1}{\nu} \aiminnorm{\Lambda^{-\alpha} \theta}_{L^2}^2.
\end{equation}
Since $\boldsymbol u$ and $\theta$ have mean zero, by Poincar\'e's inequality, we have
\begin{equation}\label{eq3.1.5}
\frac{\mathrm{d}}{\mathrm{d}t} \aiminnorm{\boldsymbol u}_{L^2}^2 + \nu \lambda_1^{2\alpha} \aiminnorm{\boldsymbol u}_{L^2}^2 \leq \frac{1}{\nu \lambda_1^{2\alpha}} \aiminnorm{\theta}_{L^2}^2
.
\end{equation}
Now, integrating in time and using \eqref{eq3.1.1}, we obtain in the case that $\nu \lambda_1^{2\alpha} \not= \kappa\lambda_1^{2\beta}$,
\begin{equation}\label{eq3.1.6}
\aiminnorm{\boldsymbol u(t)}_{L^2}^2 \leq e^{- \nu \lambda_1^{2\alpha} t} \aiminnorm{\boldsymbol u_0}_{L^2}^2
+  \frac{1}{\nu \lambda_1^{2\alpha}}
\left|\frac{e^{-\nu \lambda_1^{2\alpha}t}- e^{- \kappa \lambda_1^{2\beta}t}}{\nu \lambda_1^{2\alpha}-\kappa \lambda_1^{2\beta}} \right|\left(\aiminnorm{\theta_0}_{L^2}^2 - \frac{\aiminnorm{f}_{L^2}^2}{\kappa^2 \lambda_1^{4\beta}}\right)
+ \frac{\aiminnorm{f}_{L^2}^2}{(\kappa \lambda_1^{2\beta}\nu \lambda_1^{2\alpha})^2},
\end{equation}
and in the case that  $\nu \lambda_1^{2\alpha} = \kappa \lambda_1^{2\beta}$,
\begin{equation} \label{eq3.1.6.0}
\aiminnorm{\boldsymbol u(t)}_{L^2}^2 \leq e^{- \nu \lambda_1^{2\alpha} t} \aiminnorm{\boldsymbol u_0}_{L^2}^2
+  e^{- \nu \lambda_1^{2\alpha} t}\frac{t}{\nu \lambda_1^{2\alpha}}
\left(\aiminnorm{\theta_0}_{L^2}^2 - \frac{\aiminnorm{f}_{L^2}^2}{\kappa^2 \lambda_1^{4\beta}}\right)
+ \frac{\aiminnorm{f}_{L^2}^2}{\kappa \lambda_1^{2\beta}(\nu \lambda_1^{2\alpha})^2},
\end{equation}
which shows the uniform estimate \eqref{eq3.1.4} for some
$t_2^* >0$ large enough. Furthermore, integrating \eqref{eq3.1.9} in time gives
\begin{equation}\label{eq3.1.7}
\aiminnorm{\boldsymbol u(t+1)}_{L^2}^2 + \nu \int_{t}^{t+1}  \aiminnorm{\Lambda^{\alpha} \boldsymbol u}_{L^2}^2 {\mathrm{d}s} \leq \aiminnorm{\boldsymbol u(t)}_{L^2}^2+ \frac{1}{\nu\lambda_1^{2\alpha}} \aiminnorm{\theta}_{L^2}^2.
\end{equation}
Therefore, the time average estimate
\eqref{eq3.1.8} of $\aiminnorm{\Lambda^{\alpha} \boldsymbol u}^2$ follows from
\eqref{eq3.1.0} and \eqref{eq3.1.4}.
\end{proof}

\subsection{$L^2$-estimate for the vorticity $w$} \label{sec3.3}
In order to proceed to find a uniform $H^s$-estimate for $(\theta, \boldsymbol u)$ in the next subsection, we also need a uniform $L^p$-estimate for the velocity $\boldsymbol u$. However, we could not simply multiply $\aiminabs{\boldsymbol u}^{p-1}\boldsymbol u$ to show the $L^p$-estimate because of the troublesome pressure term. Here, we aim to show the uniform $H^1$-estimate for the velocity $\boldsymbol u$ by proving the uniform $L^2$-estimate for the vorticity $\omega$,
which satisfies
\begin{equation}\begin{cases}\label{eq3.3.1}
\partial_t\omega + \boldsymbol u\cdot \nabla \omega + \nu(-\Delta)^\alpha \omega= \partial_1\theta , \\
\boldsymbol u = \nabla^{\perp} \Delta^{-1} \omega,  \\
\end{cases}\end{equation}
where $\nabla^{\perp}= (-\partial_2, \partial_1)$.
\begin{proposition}[Existence of absorbing ball in $L^2$ for $\omega$]\label{prop3.3}
 Under the assumptions of Theorem~\ref{thm2.1}, there exists $t_3^{*}=t_3^{*}(\aiminnorm{\omega_0}_{L^2}, \aiminnorm{\theta_0}_{L^2})  > 0$, such that
\begin{equation} \label{eq3.3.0}
\aiminnorm{\omega(t)}_{L^2},\; \aiminnorm{\boldsymbol u}_{H^1} \leq C, \qquad \forall t \geq t_3^{*},
\end{equation}
and
\begin{equation}\label{eq3.3.3}
\int_{t}^{t+1}  \aiminnorm{\Lambda^{\alpha} \omega}_{L^2}^2 {\mathrm{d}s} \leq C, \qquad \forall t \geq t_3^{*}.
\end{equation}

\end{proposition}
\begin{proof} Taking the inner product of the equation \eqref{eq3.3.1}$_1$ with $\omega$ in $L^2$, we have
$$\frac{1}{2} \frac{\mathrm{d}}{\mathrm{d}t} \aiminnorm{\omega}_{L^2}^2 + \nu \aiminnorm{\Lambda^{\alpha} \omega}_{L^2}^2 =  \aimininner{\partial_1\theta}{\omega} \leq \frac{1}{2\nu} \aiminnorm{\Lambda^{1-\alpha} \theta}_{L^2}^2+ \frac{\nu}{2}\aiminnorm{\Lambda^{\alpha} \omega}_{L^2}^2.$$
Hence,
\begin{equation} \label{eq3.3.10}
\frac{\mathrm{d}}{\mathrm{d}t} \aiminnorm{\omega}_{L^2}^2 + \nu \aiminnorm{\Lambda^{\alpha} \omega}_{L^2}^2 \leq  \frac{1}{\nu} \aiminnorm{\Lambda^{1-\alpha} \theta}_{L^2}^2.
\end{equation}
Since $\omega$ has mean zero, Poincar\'e's inequality implies
\begin{equation} \label{eq3.3.11}
\frac{\mathrm{d}}{\mathrm{d}t} \aiminnorm{\omega}_{L^2}^2 + \nu \lambda_1^{2\alpha} \aiminnorm{\omega}_{L^2}^2 \leq  \frac{1}{\nu} \aiminnorm{\Lambda^{1-\alpha}\theta}_{L^2}^2.
\end{equation}
By the assumption \eqref{eq1.2}, we have $1-\alpha \leq \beta$ and using the interpolation inequality, we deduce from Proposition \ref{prop3.1} that
\[
\int_{t}^{t+1} \aiminnorm{\Lambda^{1-\alpha}\theta}_{L^2}^2 {\mathrm{d}s} \leq C,\qquad \forall t \geq t_1^*.
\]
Hence, applying Lemma \ref{lemma 2.2} with $y=\aiminnorm{\omega}_{L^2}^2$, $g=\frac{1}{\nu} \aiminnorm{\Lambda^{1-\alpha}\theta}_{L^2}^2$, and $\lambda=\nu \lambda_1^{2\alpha}$, we  obtain the uniform estimate \eqref{eq3.3.0} for $\omega$ and  also for $\boldsymbol u$ by the Biot-Savart law.
Furthermore, integrating \eqref{eq3.3.10} in time yields
\begin{equation}\label{eq3.3.2}
\aiminnorm{\omega(t+1)}_{L^2}^2 + \nu \int_{t}^{t+1}  \aiminnorm{\Lambda^{\alpha} \omega}_{L^2}^2 {\mathrm{d}s} \leq \aiminnorm{\omega(t)}_{L^2}^2+ \frac{1}{\nu}
\int_{t}^{t+1} \aiminnorm{\Lambda^{1-\alpha}\theta}_{L^2}^2 {\mathrm{d}s}.
\end{equation}
Therefore, the time average estimate
\eqref{eq3.3.3} of $\aiminnorm{\Lambda^{\alpha} \omega}^2$ follows.
\end{proof}
As an immediate consequence of Proposition \ref{prop3.3} and the Sobolev embedding theorem, we have a uniform $L^p$-estimate for $\boldsymbol u$ and a time average estimate of $\aiminnorm{\Lambda^{1+\alpha} \boldsymbol u}^2$,
that is for all $1<p<\infty$,
\begin{equation} \label{eq3.4.2}
\aiminnorm{\boldsymbol u}_{L^p}\leq C(p), \qquad
\int_{t}^{t+1} \aiminnorm{\Lambda^{1+\alpha} \boldsymbol u}^2 {\mathrm{d}s} \leq C,
\qquad \forall t\geq t_3^*,
\end{equation}
where the constant $C(p)>0$ only depends on $p$, but is independent of the time $t$ and the initial data $\boldsymbol u_0$ and $\theta_0$.
\subsection{$H^s$-estimate of $\theta$ and $\boldsymbol u$} \label{sec3.4}
\begin{proposition} [Existence of absorbing ball in $H^s$ for $(\theta, \boldsymbol u)$]\label{prop3.4}
Under the assumptions of Theorem~\ref{thm2.1},
there exists $t_4^{*}=t_4^{*}(\aiminnorm{\theta_0}_{L^2}, \aiminnorm{\boldsymbol u_0}_{L^2})  > 0$, such that
\begin{equation} \label{eq3.4.0.1}
\aiminnorm{\theta(t)}_{H^{s_1}} \leq C, \qquad \aiminnorm{\boldsymbol u}_{H^{s_2}} \leq C, \qquad \forall t \geq t_4^{*},
\end{equation}
and
\begin{equation}\label{eq3.4.0.2}
\int_{t}^{t+1}  \aiminnorm{\Lambda^{s_1+\beta} \theta}_{L^2}^2 {\mathrm{d}s} \leq C,
\qquad \int_{t}^{t+1}  \aiminnorm{\Lambda^{s_2+\alpha} \boldsymbol u}_{L^2}^2 {\mathrm{d}s}
\leq C, \qquad \forall t \geq t_4^{*}.
\end{equation}
\end{proposition}
\begin{proof}
Taking the inner product of  the equation \eqref{eq1.1.1}$_3$ with $\Lambda^{2{s_1}}\theta$ in $L^2$, we have
\begin{equation}
\begin{split}
\frac{1}{2} \frac{\mathrm{d}}{\mathrm{d}t} \aiminnorm{\Lambda^{s_1} \theta}_{L^2}^2 + \aimininner{\boldsymbol u \cdot \nabla \theta}{\Lambda^{2{s_1}}\theta}+ \kappa \aiminnorm{\Lambda^{s_1+\beta} \theta}_{L^2}^2 &= \aimininner{f}{\Lambda^{2s_1}\theta} \\
&\leq \frac{C}{\kappa}\aiminnorm{\Lambda^{s_1-\beta}f}_{L^2}^2+
\frac{\kappa}{12}\aiminnorm{\Lambda^{s_1+\beta}\theta}_{L^2}^2.
\end{split}
\end{equation}
Since $\boldsymbol u$ is divergence free, then $\boldsymbol u \cdot \nabla \theta= \nabla \cdot (\boldsymbol u \theta)$ and according to Lemma~\ref{lem2.4}, we have
\begin{equation} \label{eq3.4.1}
\begin{split}
\aiminabs {\aimininner{\boldsymbol u \cdot \nabla \theta}{\Lambda^{2s_1}\theta}}
&= \aiminabs {\aimininner{\nabla \cdot (\boldsymbol u \theta)}{\Lambda^{2s_1}\theta}} \\
&\leq \aiminnorm{\Lambda^{s_1+\beta_1}\theta}_{L^2}
\aiminnorm{\Lambda^{s_1+1-\beta_1}(\boldsymbol u \theta)}_{L^2}  \\
&\leq C\aiminnorm{\Lambda^{s_1+\beta_1}\theta}_{L^2}
\left(\aiminnorm{\Lambda^{s_1+1-\beta_1}\boldsymbol u}_{L^{p_\theta}}\aiminnorm{\theta}_{L^{q_\theta}}+
\aiminnorm{\Lambda^{s_1+1-\beta_1}\theta}_{L^{p_\theta}}\aiminnorm{\boldsymbol u}_{L^{q_\theta}}\right),
\end{split}
\end{equation}
for some $\beta_1>0,\, {p_\theta},\, {q_\theta}>2$, which are determined later,  such that $\beta_1<\beta$ and $1/p_\theta+1/q_\theta=1/2$. In order to determine $p_\theta, q_\theta$ and $\beta_1$, we first choose $r_1>0$ such that
\begin{equation}
r_1=\begin{cases}
s_1, \quad & 2\mathrm{max}\{1-\alpha, 1-\beta\} < s_1 < 1,\\
\text{any number in} \hspace{2pt} (2\mathrm{max} \{1-\alpha, 1-\beta\}, 1),&\qquad s_1\geq 1.
\end{cases}
\end{equation}
Then, the Sobolev embedding theorem implies that
$$\theta_0 \in H^{s_1} \subset H^{r_1} \subset L^{q_\theta},$$
where $q_\theta$ is chosen such that
$$\frac{1}{q_{\theta}} = \frac{1-r_1}{2} < \mathrm{min}\{\alpha, \beta \}-\frac{1}{2}.$$
Therefore, by the uniform $L^p$-estimates \eqref{eq3.1.2} and \eqref{eq3.4.2} for $(\boldsymbol u, \theta)$, we infer from \eqref{eq3.4.1} that for $t > \mathrm{max} \{t_1^*, t_2^* \}$,
\begin{equation} \label{eq3.4.3}
\begin{split}
\aiminabs {\aimininner{\boldsymbol u \cdot \nabla \theta}{\Lambda^{2s_1}\theta}}
&\leq C\aiminnorm{\Lambda^{s_1+\beta_1}\theta}_{L^2}
\aiminnorm{\Lambda^{s_1+1-\beta_1}\boldsymbol u}_{L^{p_\theta}}+
C\aiminnorm{\Lambda^{s_1+\beta_1}\theta}_{L^2}\aiminnorm{\Lambda^{s_1+1-\beta_1}\theta}_{L^{p_\theta}} \\
& =: I_1+I_2.
\end{split}
\end{equation}
Now, we set $s_1+2-2/p_\theta-\beta_1=s_1+\beta_1$, that is $\beta_1 = 1/2+ 1/q_{\theta} < \text{min}\{\alpha, \beta\}$, then the Sobolev embedding theorem implies that $H^{s_1+2-\frac{2}{p_\theta}-\beta_1}(\Omega) \hookrightarrow H^{{s_1+1-\beta_1},{p_\theta}}(\Omega)$ and the term $I_1$ can be estimated as
\begin{equation} \label{eq3.4.4}
\begin{split}
I_1&\leq C \aiminnorm{\Lambda^{s_1+\beta_1}\theta}_{L^2}
\aiminnorm{\Lambda^{s_1+\beta_1}\boldsymbol u}_{L^2} \\
&\leq C\aiminnorm{\Lambda^{s_1+\beta}\theta}_{L^2}^{\frac{\beta_1}{\beta}}
\aiminnorm{\Lambda^{s_1}\theta}_{L^2}^{1-\frac{\beta_1}{\beta}} \aiminnorm{\Lambda^{s_1+\alpha}\boldsymbol u}_{L^2}^{\frac{\beta_1}{\alpha}}
\aiminnorm{\Lambda^{s_1}\boldsymbol u}_{L^2}^{1-\frac{\beta_1}{\alpha}},
\end{split}
\end{equation}
where we used the interpolation inequality for the tuples $(s_1, s_1+\beta_1, s_1+\beta)$ and $(s_1, s_1 + \beta_1, s_1 + \alpha)$. Since $0 < s_1 \leq s_2$, applying Poincar\'e's inequality,
we find
\begin{equation} \label{eq3.4.4.1}
\aiminnorm{\Lambda^{s_1+\alpha}\boldsymbol u}_{L^2} \leq
C\aiminnorm{\Lambda^{s_2+\alpha}\boldsymbol u}_{L^2},  \qquad
\aiminnorm{\Lambda^{s_1}\boldsymbol u}_{L^2} \leq
C\aiminnorm{\Lambda^{s_2}\boldsymbol u}_{L^2}.
\end{equation}
Hence, using Young's inequality, we find
\begin{equation} \label{eq3.4.4.2}
\begin{split}
I_1 &\leq C\aiminnorm{\Lambda^{s_1+\beta}\theta}_{L^2}^{\frac{\beta_1}{\beta}}
\aiminnorm{\Lambda^{s_1}\theta}_{L^2}^{1-\frac{\beta_1}{\beta}}
\aiminnorm{\Lambda^{s_2+\alpha}\boldsymbol u}_{L^2}^{\frac{\beta_1}{\alpha}}
\aiminnorm{\Lambda^{s_2}\boldsymbol u}_{L^2}^{1-\frac{\beta_1}{\alpha}} \\
& \leq \frac{\kappa}{12}\aiminnorm{\Lambda^{s_1+\beta}\theta}_{L^2}^2
+ \frac{C}{\kappa^a}\aiminnorm{\Lambda^{s_1}\theta}_{L^2}^2
+\frac{\nu}{4}\aiminnorm{\Lambda^{s_2+\alpha}\boldsymbol u}_{L^2}^2
+\frac{C}{\nu^b}\aiminnorm{\Lambda^{s_2}\boldsymbol u}_{L^2}^2,
\end{split}
\end{equation}
where $a=\beta_1/(\beta-\beta_1)$ and $b=\beta_1/(\alpha-\beta_1)$. For the term $I_2$ as for $I_1$, we deduce from the same Sobolev embedding and interpolation inequality that
\begin{equation}  \label{eq3.4.5}
\begin{split}
I_2
\leq C \aiminnorm{\Lambda^{s_1+\beta_1}\theta}_{L^2}^2
&\leq C\aiminnorm{\Lambda^{s_1+\beta}\theta}_{L^2}^{\frac{2\beta_1}{\beta}}
\aiminnorm{\Lambda^{s_1}\theta}_{L^2}^{2-\frac{2\beta_1}{\beta}}  \\
&\leq \frac{\kappa}{12} \aiminnorm{\Lambda^{s_1+\beta}\theta}_{L^2}^2
+\frac{C}{\kappa^a}\aiminnorm{\Lambda^{s_1}\theta}_{L^2}^2.
\end{split}
\end{equation}
Finally, we arrive at a differential inequality for $\theta$:
\begin{equation} \label{eq3.4.6}
\begin{split}
\frac{1}{2} \frac{\mathrm{d}}{\mathrm{d}t} \aiminnorm{\Lambda^{s_1} \theta}_{L^2}^2 + \kappa \aiminnorm{\Lambda^{s_1+\beta} \theta}_{L^2}^2
\leq & \frac{\kappa}{4} \aiminnorm{\Lambda^{s_1+\beta}\theta}_{L^2}^2
+\frac{\nu}{4} \aiminnorm{\Lambda^{s_2+\alpha}\boldsymbol u}_{L^2}^2
+\frac{C}{\kappa^a}\aiminnorm{\Lambda^{s_1}\theta}_{L^2}^2 \\
&+\frac{C}{\kappa}\aiminnorm{\Lambda^{s_1-\beta}f}_{L^2}^2+\frac{C}{\nu^b}\aiminnorm{\Lambda^{s_2}\boldsymbol u}_{L^2}^2.
\end{split}
\end{equation}

Taking the inner product of  the equation of \eqref{eq1.1.1}$_1$ with $\Lambda^{2s_2}\boldsymbol u$ in $L^2$, we have
\begin{equation}\label{eq3.5.1}
\frac{1}{2} \frac{\mathrm{d}}{\mathrm{d}t} \aiminnorm{\Lambda^{s_2} \boldsymbol u}_{L^2}^2 + \aimininner{\boldsymbol u \cdot \nabla \boldsymbol u}{\Lambda^{2s_2}\boldsymbol u}+ \nu \aiminnorm{\Lambda^{s_2+\alpha} \boldsymbol u}_{L^2}^2 = \aimininner{\theta \boldsymbol e_2}{\Lambda^{2s_2} \boldsymbol u}.
\end{equation}
The right-hand side of \eqref{eq3.5.1} is estimated by
\[
\aiminabs{\aimininner{\theta \boldsymbol e_2}{\Lambda^{2s_2} \boldsymbol u}}
 = \aiminabs{\aimininner{\Lambda^{s_1+\beta}\theta \boldsymbol e_2}{\Lambda^{2s_2-s_1-\beta} \boldsymbol u}}
 \leq \aiminnorm{\Lambda^{s_1+\beta}\theta}_{L^2} \aiminnorm{\Lambda^{2s_2-s_1-\beta}\boldsymbol u}_{L^2},
\]
which by using the interpolation inequality for the tuple $(0, 2s_2-s_1-\beta, s_2+\alpha)$ because of the assumption $0 \leq s_2-s_1 < \alpha + \beta$, is bounded by
\[
\aiminnorm{\Lambda^{s_1+\beta}\theta}_{L^2}
\aiminnorm{\Lambda^{s_2+\alpha}\boldsymbol u}_{L^2}^{1-\frac{s_1+\alpha+\beta-s_2}{s_2+\alpha}}
\aiminnorm{\boldsymbol u}_{L^2} ^{\frac{s_1+\alpha+\beta-s_2}{s_2+\alpha}},
\]
and finally applying Young's inequality, and since $\aiminnorm{\boldsymbol u}_{L^2}$ is uniformly bounded by Proposition~\ref{prop3.2}, we obtain that for $t > t^*_2$,
\begin{equation} \label{eq3.5.1.1}
\aiminabs{\aimininner{\theta \boldsymbol e_2}{\Lambda^{2s_2} \boldsymbol u}}
 \leq \frac{\kappa}{4} \aiminnorm{\Lambda^{s_1+\beta}\theta}_{L^2}^2
+ \frac{\nu}{8} \aiminnorm{\Lambda^{s_2+\alpha}\boldsymbol u}_{L^2}^2
+ \frac{C}{\kappa^{a'} \nu^{\frac{1}{a'}-1}},
\end{equation}
where $a'= (s_1+\alpha+\beta-s_2)/(s_2+\alpha)$. We now estimate the nonlinear term in \eqref{eq3.5.1}. Similar to \eqref{eq3.4.1} and using the $L^p$-estimate \eqref{eq3.4.2} of $\boldsymbol u$ and Lemma \ref{lem2.4}, we first have that for $t > \mathrm{max} \{t_2^*, t_3^*\}$,
\begin{equation} \label{eqn3.5.2.2}
\begin{split}
\aiminabs{\aimininner{\boldsymbol u \cdot \nabla \boldsymbol u}{\Lambda^{2s_2}\boldsymbol u}}
&\leq C \aiminnorm{\Lambda^{s_2+\alpha_1}\boldsymbol u}_{L^2} \aiminnorm{\Lambda^{s_2+1-\alpha_1} (\boldsymbol u \otimes \boldsymbol u)}_{L^2} \\
&\leq C \aiminnorm{\Lambda^{s_2+\alpha_1}\boldsymbol u}_{L^2} \aiminnorm{\Lambda^{s_2+1-\alpha_1}\boldsymbol u}_{L^{p_{\boldsymbol u}}}
\aiminnorm{\boldsymbol u}_{L^{q_{\boldsymbol u}}}  \\
&\leq C \aiminnorm{\Lambda^{s_2+\alpha_1}\boldsymbol u}_{L^2} \aiminnorm{\Lambda^{s_2+1-\alpha_1}\boldsymbol u}_{L^{p_{\boldsymbol u}}}, \\
\end{split}
\end{equation}
for some $\alpha_1>0, \, {p_{\boldsymbol u}},\, {q_{\boldsymbol u}}>2$, which are determined later,  such that $\alpha_1< \alpha$ and $1/p_{\boldsymbol u}+1/q_{\boldsymbol u}=1/2$. Now, if we set $s_2+2-2/p_{\boldsymbol u}-\alpha_1=s_2+\alpha_1$, so that $\alpha_1=1/2+ 1/q_{\boldsymbol u}=1-1/p_{\boldsymbol u}$ and choose $q_{\boldsymbol u}$ large enough such that $\alpha_1 < \alpha$, then by the Sobolev embedding theorem $H^{s_2+2-\frac{2}{p_{\boldsymbol u}}-\alpha_1}(\Omega) \hookrightarrow H^{{s_2+1-\alpha_1},{p_{\boldsymbol u}}}(\Omega)$ and the interpolation inequality for the tuple $(s_2, s_2+\alpha_1, s_2 + \alpha)$ and Young's inequality, we arrive at
\begin{equation}\label{eq3.5.2}
\begin{split}
\aiminabs{\aimininner{\boldsymbol u \cdot \nabla \boldsymbol u}{\Lambda^{2s_2}\boldsymbol u}}
&\leq C \aiminnorm{\Lambda^{s_2+\alpha_1}\boldsymbol u}_{L^2}^2
\leq C  \aiminnorm{\Lambda^{s_2+\alpha}\boldsymbol u}_{L^2}^{\frac{2\alpha_1}{\alpha}}  \aiminnorm{\Lambda^{s_2}\boldsymbol u}_{L^2}^{2-\frac{2\alpha_1}{\alpha}}\\
&\leq \frac{\nu}{8}  \aiminnorm{\Lambda^{s_2+\alpha}\boldsymbol u}_{L^2}^2+ \frac{C}{\nu^{b'}} \aiminnorm{\Lambda^{s_2}\boldsymbol u}_{L^2}^2,
\end{split}
\end{equation}
where $b'=\alpha_1/(\alpha-\alpha_1)$. Hence, we infer from \eqref{eq3.5.1} that
\begin{equation}\label{eq3.5.3}
\frac{1}{2} \frac{\mathrm{d}}{\mathrm{d}t} \aiminnorm{\Lambda^{s_2} \boldsymbol u}_{L^2}^2 + \nu \aiminnorm{\Lambda^{s_2+\alpha} \boldsymbol u}_{L^2}^2
\leq \frac{\kappa}{4} \aiminnorm{\Lambda^{s_1+\beta}\theta}_{L^2}^2 +\frac{\nu}{4}  \aiminnorm{\Lambda^{s_2+\alpha}\boldsymbol u}_{L^2}^2+ \frac{C}{\nu^{b'}} \aiminnorm{\Lambda^{s_2}\boldsymbol u}_{L^2}^2+\frac{C}{\kappa^{a'} \nu^{\frac{1}{a'}-1}}.
\end{equation}

Summing the differential inequalities \eqref{eq3.4.6} and \eqref{eq3.5.3} together, we obtain that for $t > \mathrm{max} \{t_1^*, t_2^*, t_3^*\}$,
\begin{equation} \label{eq3.5.4}
\begin{split}
\frac{\mathrm{d}}{\mathrm{d}t} &\big(\aiminnorm{\Lambda^{s_1} \theta}_{L^2}^2+\aiminnorm{\Lambda^{s_2} \boldsymbol u}_{L^2}^2\big) + \kappa \aiminnorm{\Lambda^{s_1+\beta} \theta}_{L^2}^2  + \nu \aiminnorm{\Lambda^{s_2+\alpha} \boldsymbol u}_{L^2}^2  \\
&\leq C\left(\frac{1}{\kappa^a}\aiminnorm{\Lambda^{s_1}\theta}_{L^2}^2 + (\frac{1}{\nu^b}+\frac{1}{\nu^{b'}})\aiminnorm{\Lambda^{s_2}\boldsymbol u}_{L^2}^2\right)+
\frac{C}{\kappa}\aiminnorm{\Lambda^{s_1-\beta}f}_{L^2}^2+\frac{C}{\kappa^{a'} \nu^{\frac{1}{a'}-1}}.
\end{split}
\end{equation}
Here, starting with $s_1=s_1^{(1)} > 2\mathrm{max} \{1-\alpha, 1-\beta\}$,
$s_1^{(1)} \leq \beta $ and $s_2=s_2^{(1)} =1$, which implies that $s_2^{(1)} -s_1^{(1)} < \mathrm{min} \{2\alpha-1, 2\beta-1\} < \alpha + \beta$, then by Poincar\'e's inequality and equations
\eqref{eq3.1.3}, \eqref{eq3.4.2}, for $t > t^* := \mathrm{max} \{t_1^*, t_2^*,t_3^*\}$, we have
\begin{equation} \label{eq3.5.4.1}
\int_{t}^{t+1} \aiminnorm{\Lambda^{s_1^{(1)}} \theta}_{L^2}^2 {\mathrm{d}s} \leq C,
\qquad
\int_{t}^{t+1} \aiminnorm{\Lambda^{s_2^{(1)}} \boldsymbol u}_{L^2}^2 {\mathrm{d}s} \leq C.
\end{equation}
Hence, applying the uniform Gronwall lemma to \eqref{eq3.5.4} and using \eqref{eq3.5.4.1}, we obtain that for $t>t^* + 1$,
\begin{equation} \label{eq3.5.5}
\aiminnorm{\Lambda^{s_1^{(1)}} \theta(t)}_{L^2} \leq C,
\qquad
\aiminnorm{\Lambda^{s_2^{(1)}} \boldsymbol u(t)}_{L^2} \leq C,
\end{equation}
Furthermore, integrating \eqref{eq3.5.4} and using \eqref{eq3.5.5},  we find that for $t>t^* + 1$,
\begin{equation} \label{eq3.5.7}
\int_{t}^{t+1} \aiminnorm{\Lambda^{s_1^{(1)}+\beta} \theta}_{L^2}^2 {\mathrm{d}s} \leq C,
\qquad
\int_{t}^{t+1} \aiminnorm{\Lambda^{s_2^{(1)}+\alpha} \boldsymbol u}_{L^2}^2 {\mathrm{d}s} \leq C.
\end{equation}
Then, we iterate with $s_1=s_1^{(2)} = s_1^{(1)}+\beta$ and let $s_2=s_2^{(2)}$ be any number that satisfies
$s_2^{(2)}-s_1^{(2)} < \alpha+\beta$ and $s_2^{(1)} < s_2^{(2)} \leq s_2^{(1)}+ \alpha$. Hence, we have from \eqref{eq3.5.7} that for $t > t^*+1$,
\begin{equation} \label{eq3.5.8}
\int_{t}^{t+1} \aiminnorm{\Lambda^{s_1^{(2)}} \theta}_{L^2}^2 {\mathrm{d}s} \leq C,
\qquad
\int_{t}^{t+1} \aiminnorm{\Lambda^{s_2^{(2)}} \boldsymbol u}_{L^2}^2 {\mathrm{d}s} \leq C.
\end{equation}
Applying the uniform Gronwall lemma to \eqref{eq3.5.4} again, this time using \eqref{eq3.5.8}, we obtain for $t>t^* + 2$,
\begin{equation} \label{eq3.5.8.1}
\aiminnorm{\Lambda^{s_1^{(2)}} \theta(t)}_{L^2} \leq C,
\qquad
\aiminnorm{\Lambda^{s_2^{(2)}} \boldsymbol u(t)}_{L^2} \leq C,
\end{equation}
and
\begin{equation} \label{eq3.5.10}
\int_{t}^{t+1} \aiminnorm{\Lambda^{s_1^{(2)}+\beta} \theta}_{L^2}^2 {\mathrm{d}s} \leq C,
\qquad
\int_{t}^{t+1} \aiminnorm{\Lambda^{s_2^{(2)}+\alpha} \boldsymbol u}_{L^2}^2 {\mathrm{d}s} \leq C.
\end{equation}
Therefore, with a bootstrapping argument, for any given real numbers $s_1, s_2$ that satisfy $s_1 > 2\mathrm{max} \{1-\alpha, 1-\beta\}$, $s_2 \geq 1$ and
$s_2-s_1 < \alpha+\beta$,
\eqref{eq3.4.0.1} and \eqref{eq3.4.0.2} are proved. In addition,
for fixed $T > 0$,
\begin{equation} \label{eq3.5.11}
\int_{0}^{T} \aiminnorm{\Lambda^{s_1+\beta} \theta}_{L^2}^2 {\mathrm{d}s} < \infty,
\qquad
\int_{0}^{T} \aiminnorm{\Lambda^{s_2+\alpha} \boldsymbol u}_{L^2}^2 {\mathrm{d}s} < \infty.
\end{equation}
By \eqref{eq3.5.8.1}, given $(\theta_0, \boldsymbol u_0) \in H^{s_1} \times H^{s_2}$, then for some $t > 0$ large enough, the solution $(\theta(t), \boldsymbol u(t))$ of system \eqref{eq1.1.1} belongs to the space $H^{s_3} \times H^{s_4}$ for some $s_3=s_1+\beta$ and some $s_4$ such that $0 \leq s_4- s_3 < \alpha+\beta$ and $s_2 < s_4 \leq s_2+ \alpha$. By the Sobolev compactness embedding theorem in \cite{Tem88}, the inclusion map
$H^{s_3} \times H^{s_4} \mapsto H^{s_1} \times H^{s_2}$ is compact. Thus, for any $s_1$ and $s_2$ satisfy \eqref{cond1} and \eqref{cond2}, the solution operator $S(t)$ defined by $S(t)(\theta_0, \boldsymbol u_0)= (\theta(t), \boldsymbol u(t))$ is a compact operator in the space $H^{s_1} \times H^{s_2}$ for some $t > 0$ large enough.
\end{proof}
\section{Continuity} \label{sec4}
\subsection{Continuity with respect to $t$} \label{sec4.1}
In this section, we check that the solution operators
$\{S(t), \forall t \geq 0\}$ are continuous in the space $H^{s_1} \times H^{s_2}$ with respect to $t$. In order to achieve this, we recall the following lemma which is a particular case of a general interpolation theorem in \cite{LM72}. A proof can be found in \cite{Tem84}.

\begin{lemma} \label{lemma4.1}
Let $V$, $H$, $V'$ be three Hilbert spaces such that
$$V \subset H=H' \subset V'$$
where $H'$ is the dual space of $H$ and $V'$ is the dual space of $V$.\\
If a function $u$ belongs to $L^2(0,T; V)$ and its derivative $u'$ belongs to $L^2(0,T; V')$, then $u$ is almost everywhere equal to a function continuous from $[0,T]$ into $H$.
\end{lemma}
\begin{proposition} \label{prop4.1}
Under the assumptions of Theorem~\ref{thm2.1}, the solutions of
Boussinesq system \eqref{eq1.1.1} satisfy $(\theta, \boldsymbol u) \in \mathcal{C}([0,T]; H^{s_1}) \times\mathcal{C}([0,T]; H^{s_2})$.
\end{proposition}
\begin{proof} For any fixed $T >0$, the uniform estimates proved
in Proposition \ref{prop3.4} ensure that
$$(\theta, \boldsymbol u) \in L^2(0,T; H^{s_1+\beta}) \times L^2(0,T; H^{s_2+\alpha}),$$ where $s_1$ and $s_2$ satisfy \eqref{cond1} and \eqref{cond2}.
Hence
$$(\Lambda^{s_1}\theta, \Lambda^{s_2} \boldsymbol u) \in L^2(0,T; H^{\beta}) \times L^2(0,T; H^{\alpha}).$$
Our goal is to show that $(\theta, \boldsymbol u) \in \mathcal{C}([0,T]; H^{s_1}) \times\mathcal{C}([0,T]; H^{s_2})$, that is $(\Lambda^{s_1}\theta, \Lambda^{s_2}\boldsymbol u) \in \mathcal{C}([0,T]; L^2) \times \mathcal{C}([0,T]; L^2)$,
and by Lemma \ref{lemma4.1}, it suffices to show that
$(\Lambda^{s_1}\theta_t, \Lambda^{s_2}\boldsymbol u_t) \in L^2([0,T]; H^{-\beta}) \times L^2([0,T]; H^{-\alpha})$.
For any $(\psi, \boldsymbol \varphi) \in H^{\beta} \times (H^{\alpha})^2$, we derive from \eqref{eq1.1.1} that
\begin{equation}\begin{cases}\label{eq4.1}
\aimininner{\Lambda^{s_1}\theta_t}{\psi}+ \aimininner{\Lambda^{s_1}(\boldsymbol u \cdot \nabla \theta)}{\psi} + \kappa \aimininner{\Lambda^{{s_1}+2\beta}\theta}{\psi}=
\aimininner{\Lambda^{s_1} f}{\psi}, \\
\aimininner{\Lambda^{s_2}\boldsymbol u_t}{\boldsymbol \varphi}+ \aimininner{\Lambda^{s_2}(\boldsymbol u \cdot \nabla \boldsymbol u)}{\boldsymbol \varphi} + \nu \aimininner{\Lambda^{{s_2}+2\alpha}\boldsymbol u}{\boldsymbol \varphi}=
\aimininner{\Lambda^{s_2} \theta \boldsymbol e_2}{\boldsymbol \varphi}.
\end{cases}\end{equation}
Then, by the Cauchy-Schwarz inequality,
\begin{equation}\begin{cases}\label{eq4.2}
\aiminabs{\aimininner{\Lambda^{s_1}\theta_t}{\psi}} \leq \aiminnorm{\Lambda^{{s_1}-\beta}(\boldsymbol u \cdot \nabla \theta)}_{L^2} \aiminnorm{\psi}_{H^{\beta}}
+\kappa \aiminnorm{\Lambda^{{s_1}+\beta}\theta}_{L^2} \aiminnorm{\psi}_{H^{\beta}}
+\aiminnorm{\Lambda^{{s_1}-\beta} f}_{L^2} \aiminnorm{\psi}_{H^{\beta}}, \\\aiminabs{\aimininner{\Lambda^{s_2}\boldsymbol u_t}{\boldsymbol \varphi}} \leq \aiminnorm{\Lambda^{{s_2}-\alpha}(\boldsymbol u \cdot \nabla \boldsymbol u)}_{L^2} \aiminnorm{\boldsymbol \varphi}_{H^{\alpha}}
+\nu \aiminnorm{\Lambda^{{s_2}+\alpha}\boldsymbol u}_{L^2} \aiminnorm{\boldsymbol \varphi}_{H^{\alpha}}
+\aiminnorm{\Lambda^{{s_2}-\alpha} \theta}_{L^2} \aiminnorm{\boldsymbol \varphi}_{H^{\alpha}},
\end{cases}\end{equation}
which implies that,
\begin{equation}\begin{cases} \label{eq4.3}
\aiminnorm{\Lambda^{s_1}\theta_t}_{H^{-\beta}} \leq \aiminnorm{\Lambda^{{s_1}-\beta}(\boldsymbol u \cdot \nabla \theta)}_{L^2}+\kappa \aiminnorm{\Lambda^{{s_1}+\beta}\theta}_{L^2} +\aiminnorm{\Lambda^{{s_1}-\beta} f}_{L^2}, \\
\aiminnorm{\Lambda^{s_2}\boldsymbol u_t}_{H^{-\alpha}} \leq \aiminnorm{\Lambda^{{s_2}-\alpha}(\boldsymbol u \cdot \nabla \boldsymbol u)}_{L^2}+\nu \aiminnorm{\Lambda^{{s_2}+\alpha}\boldsymbol u}_{L^2}
+\aiminnorm{\Lambda^{{s_2}-\alpha} \theta}_{L^2}.
\end{cases}\end{equation}
We now estimate the nonlinear term $\boldsymbol u \cdot \nabla \theta$ and $\boldsymbol u \cdot \nabla \boldsymbol u$ in the right-hand side of \eqref{eq4.3}. Since $\boldsymbol u$ is divergence free, we have
\begin{equation} \label{eq4.1.3}
\aiminnorm{\Lambda^{{s_1}-\beta}(\boldsymbol u \cdot \nabla \theta)}_{L^2}
=\aiminnorm{\Lambda^{{s_1}-\beta}\nabla \cdot (\boldsymbol u\theta)}_{L^2}
\leq \aiminnorm{\Lambda^{1+{s_1}-\beta} (\boldsymbol u\theta)}_{L^2}.
\end{equation}
We now choose $r>0$ such that
\begin{equation}
r=\begin{cases}
s_1, \quad & 2\text{max} \{1-\alpha, 1-\beta\} < s_1 < 1,\\
\text{any number in} \hspace{2pt} (2\mathrm{max} \{1-\alpha, 1-\beta\}, 1),&\qquad s_1\geq 1.
\end{cases}
\end{equation}
Then,
\[
\theta_0 \in H^{s_1} \subset H^r \subset L^{p_1},\qquad\boldsymbol u_0 \in H^{s_2} \subset H^1 \subset L^{p_2},
\]
where $$\frac{1}{p_1}= \frac{1-r}{2} < \text{min} \{\alpha-\frac{1}{2}, \beta-\frac{1}{2}\} \leq \frac{1}{2}, \qquad \displaystyle p_2 = \frac{2}{2\beta-1}.$$
Hence, by \eqref{eq3.1.2} and  \eqref{eq3.4.2},
$\theta \in L^{\infty}([0, +\infty); L^{p_1})$ and $\boldsymbol u \in L^{\infty}([0, +\infty); L^{p_2})$. Therefore
\begin{equation} \label{eq4.1.5}
\aiminnorm{\Lambda^{1+s_1-\beta} (\boldsymbol u\theta)}_{L^2}
\leq C\aiminnorm{\theta}_{L^{p_1}}\aiminnorm{\Lambda^{1+s_1-\beta} \boldsymbol u}_{L^{q_1}}+
C\aiminnorm{\boldsymbol u}_{L^{p_2}}\aiminnorm{\Lambda^{1+s_1-\beta} \theta}_{L^{q_2}},
\end{equation}
where $1/p_1+1/q_1=1/2$ and $1/p_2+1/q_2=1/2$. Now, we choose $q_1^*$ such that $$\displaystyle\frac{1}{q_1^*}+ \frac{s_2+\alpha-(1+s_1-\beta)}{2}= \frac{1}{2}.$$ Since $1/q_1= r/2$ and $s_2 > s_1$, we have
$$\displaystyle \frac{1}{q_1^*} \leq \frac{1}{2} (2-\alpha-\beta) \leq \frac{1}{2} \cdot 2\mathrm{max} \{1-\alpha, 1-\beta\} <  \frac{1}{2}r= \frac{1}{q_1}.$$ Hence, $q_1 < q_1^*$.
Thus, by Poincar\'e's inequality and the Sobolev embedding theorem,
\begin{equation} \label{eq4.1.6}
\aiminnorm{\Lambda^{1+s_1-\beta} \boldsymbol u}_{L^{q_1}}
\leq C\aiminnorm{\Lambda^{1+s_1-\beta} \boldsymbol u}_{L^{q_1^*}}
\leq C\aiminnorm{\Lambda^{s_2+\alpha} \boldsymbol u}_{L^2}.
\end{equation}
We also deduce from $1/q_2=1/2-1/p_2=1-\beta$ that
$$\displaystyle \frac{1}{q_2}+ \frac{s_1+\beta-(1+s_1-\beta)}{2}= \frac{1}{2},$$ and hence by the Sobolev embedding theorem,
\begin{equation} \label{eq4.1.7}
\aiminnorm{\Lambda^{1+s_1-\beta} \theta}_{L^{q_2}}
\leq C\aiminnorm{\Lambda^{s_1+\beta} \theta}_{L^2}.
\end{equation}
Hence, by \eqref{eq4.2}, \eqref{eq4.1.3}, \eqref{eq4.1.5},\eqref{eq4.1.6}, \eqref{eq4.1.7}, we have
$$\aiminnorm{\Lambda^{s_1} \theta_t}_{H^{-\beta}}
\leq C\aiminnorm{\theta}_{L^{p_1}}\aiminnorm{\Lambda^{s_2+\alpha} \boldsymbol u}_{L^2}
+ C(\aiminnorm{\boldsymbol u}_{L^{p_2}}+\kappa)\aiminnorm{\Lambda^{s_1+\beta} \theta}_{L^2} + C\aiminnorm{\Lambda^{s_1-\beta} f}_{L^2},
$$
which by utilizing the estimates \eqref{eq3.1.2}, \eqref{eq3.4.2} and \eqref{eq3.5.11}, shows
$$\int_{0}^T \aiminnorm{\Lambda^{s_1}\theta_t}_{H^{-\beta}}^2 {\mathrm{d}s} < \infty.$$
Similarly for the term $\boldsymbol u \cdot \nabla \boldsymbol u$, we similarly have
\begin{equation} \label{eq4.1.4}
\begin{split}
\aiminnorm{\Lambda^{{s_2}-\alpha}(\boldsymbol u \cdot \nabla \boldsymbol u)}_{L^2}
=\aiminnorm{\Lambda^{{s_2}-\alpha}\nabla \cdot (\boldsymbol u \otimes \boldsymbol u)}_{L^2}
&\leq \aiminnorm{\Lambda^{1+{s_2}-\alpha} (\boldsymbol u \otimes \boldsymbol u)}_{L^2} \\
& \leq C \aiminnorm{\boldsymbol u}_{L^{p_3}} \aiminnorm{\Lambda^{1+{s_2}-\alpha} \boldsymbol u}_{L^{q_3}},
\end{split}
\end{equation}
where $1/p_3+1/q_3=1/2$ and $q_3= 1/(1-\alpha)$.
By \eqref{eq3.4.2},
we know that $\boldsymbol u \in L^{\infty}(0, +\infty; L^{p_3})$ for $p_3 = 2/(2\alpha-1)$. Since
$$\displaystyle \frac{1}{q_3}+ \frac{s_2+\alpha-(1+s_2-\alpha)}{2}= \frac{1}{2},$$ then the Sobolev embedding theorem shows that
\begin{equation} \label{eq4.1.8}
\aiminnorm{\Lambda^{1+{s_2}-\alpha} \boldsymbol u}_{L^{q_3}}
\leq C\aiminnorm{\Lambda^{{s_2}+\alpha} \boldsymbol u}_{L^2}.
\end{equation}
Therefore, by \eqref{eq4.3}, \eqref{eq4.1.4}, \eqref{eq4.1.8} and the assumption that $0 \leq s_2-s_1 < \alpha+\beta$, we find
\begin{equation} \label{eq4.1.9}
\aiminnorm{\Lambda^{s_2}\boldsymbol u_t}_{H^{-\alpha}} \leq
C(\aiminnorm{\boldsymbol u}_{L^{p_3}}+\nu) \aiminnorm{\Lambda^{{s_2}+\alpha} \boldsymbol u}_{L^2}+ C\aiminnorm{\Lambda^{{s_1}+\beta} \theta}_{L^2},
\end{equation}
which by combining the estimates \eqref{eq3.4.2} and \eqref{eq3.5.11}, yields
$$\int_{0}^T \aiminnorm{\Lambda^{s_2}\boldsymbol u_t}_{H^{-\alpha}}^2 {\mathrm{d}s} < \infty. $$
We thus completed the proof of Proposition 4.1.
\end{proof}

\subsection{Continuity for fixed $t>0$} \label{sec4.2}
We simultaneously prove uniqueness and continuity of $S(t)$ from $H^{s_1} \times H^{s_2}$ to itself for any fixed $t>0$.
\begin{proposition} \label{prop4.1}
Under the assumptions of Theorem~\ref{thm2.1}, the solution of
Boussinesq system \eqref{eq1.1.1} is unique and the solution operator $S(t): H^{s_1} \times H^{s_2} \mapsto H^{s_1} \times H^{s_2}$ is continuous for any fixed $t>0$.
\end{proposition}
\begin{proof}
Let $(\boldsymbol u_1, \theta_1, \pi_1)$, $(\boldsymbol u_2, \theta_2, \pi_2)$ satisfy \eqref{eq1.1.1} with two initial data  $(\boldsymbol u_1^0, \theta_1^0)$,  $(\boldsymbol u_1^0, \theta_1^0)$ respectively.
Then $\boldsymbol \zeta = \boldsymbol u_1 - \boldsymbol u_2$, $\eta = \theta_1-\theta_2$, $\pi=\pi_1-\pi_2$ satisfy
\begin{equation}\begin{cases}\label{eq4.2.3}
\partial_t\boldsymbol \zeta + \boldsymbol u_1\cdot \nabla \boldsymbol \zeta + \boldsymbol \zeta \cdot \nabla \boldsymbol u_2
 + \nu(-\Delta)^\alpha \boldsymbol \zeta= - \nabla \pi + \eta \boldsymbol e_2, \\
\partial_t \eta + \boldsymbol u_1\cdot \nabla \eta + \boldsymbol \zeta \cdot \nabla \theta_2+ \kappa(-\Delta)^\beta \eta= 0.
\end{cases}\end{equation}
Taking the inner product of \eqref{eq4.2.3} with $(\Lambda^{2s_2} \boldsymbol \zeta,\Lambda^{2s_1} \eta)$ in $L^2$, and using $\nabla \cdot \Lambda^{2s_1} \boldsymbol \zeta=0$, which comes from the fact that $\Lambda^s$ and $\nabla$ commute with each other and $\boldsymbol \zeta$ is divergence free, we obtain that
\begin{equation}\begin{cases}\label{eq4.2.19}
\frac{1}{2} \frac{\mathrm{d}}{\mathrm{d}t} \aiminnorm{\Lambda^{s_2}\boldsymbol \zeta}_{L^2}^2 + \nu \aiminnorm{\Lambda^{s_2+\alpha} \boldsymbol \zeta}_{L^2}^2
= \aimininner{\eta \boldsymbol e_2}{\Lambda^{2s_2} \boldsymbol \zeta} -\aimininner{\boldsymbol u_1 \cdot \nabla \boldsymbol \zeta}{\Lambda^{2s_2}\boldsymbol \zeta}-\aimininner{\boldsymbol \zeta \cdot \nabla \boldsymbol u_2}{\Lambda^{2s_2}\boldsymbol \zeta}, \\
\frac{1}{2} \frac{\mathrm{d}}{\mathrm{d}t} \aiminnorm{\Lambda^{s_1}\eta}_{L^2}^2 + \kappa \aiminnorm{\Lambda^{s_1+\beta} \eta}_{L^2}^2
= -\aimininner{\boldsymbol u_1 \cdot \nabla \eta}{\Lambda^{2s_1}\eta}-\aimininner{\boldsymbol \zeta \cdot \nabla \theta_2}{\Lambda^{2s_1}\eta}.
\end{cases}\end{equation}
Similar to \eqref{eq3.5.1.1}, using Poincar\'e's inequality, the interpolation inequality and Young's inequality, we have
\begin{equation} \label{eq4.2.21}
\begin{split}
\aiminabs{\aimininner{\eta \boldsymbol e_2}{\Lambda^{2s_2} \boldsymbol \zeta}}
&\leq \frac{\kappa}{6} \aiminnorm{\Lambda^{s_1+\beta}\eta}_{L^2}^2
+ \frac{\nu}{6} \aiminnorm{\Lambda^{s_2+\alpha}\boldsymbol \zeta}_{L^2}^2
+ \frac{C}{\kappa^{a'} \nu^{\frac{1}{a'}-1}} \aiminnorm{\boldsymbol \zeta}_{L^2}^2 \\
& \leq \frac{\kappa}{6} \aiminnorm{\Lambda^{s_1+\beta}\eta}_{L^2}^2
+ \frac{\nu}{6} \aiminnorm{\Lambda^{s_2+\alpha}\boldsymbol \zeta}_{L^2}^2
+ \frac{C}{\kappa^{a'} \nu^{\frac{1}{a'}-1}} \aiminnorm{\Lambda^{s_2}\boldsymbol \zeta}_{L^2}^2,
\end{split}
\end{equation}
To deal with the term $\aimininner{\boldsymbol u_1 \cdot \nabla \boldsymbol \zeta}{\Lambda^{2s_2}\boldsymbol \zeta}$,
we observe that $\aimininner{\boldsymbol u_1 \cdot \nabla (\Lambda^{s_2}\boldsymbol \zeta)}{\Lambda^{s_2}\boldsymbol \zeta}=0$, so that
\begin{equation} \label{eq4.2.20.6}
\aimininner{\boldsymbol u_1 \cdot \nabla \boldsymbol \zeta}{\Lambda^{2s_2}\boldsymbol \zeta}
= \aimininner{\Lambda^{s_2}(\boldsymbol u_1 \cdot \nabla \boldsymbol \zeta)}{\Lambda^{s_2}\boldsymbol \zeta}
= \aimininner{\Lambda^{s_2}(\boldsymbol u_1 \cdot \nabla \boldsymbol \zeta)-\boldsymbol u_1 \cdot \nabla (\Lambda^{s_2}\boldsymbol \zeta)}{\Lambda^{s_2}\boldsymbol \zeta}.
\end{equation}
Noticing that $\nabla$ and $\Lambda$ commute, we find
\begin{equation}\label{eq4.2.20.1}
\begin{split}
&\aiminabs{\aimininner{\Lambda^{s_2}(\boldsymbol u_1 \cdot \nabla \boldsymbol \zeta)-\boldsymbol u_1 \cdot \nabla (\Lambda^{s_2}\boldsymbol \zeta)}{\Lambda^{s_2}\boldsymbol \zeta}}
=\aiminabs{\aimininner{\Lambda^{s_2}(\boldsymbol u_1 \cdot \nabla \boldsymbol \zeta)-\boldsymbol u_1 \cdot (\Lambda^{s_2}\nabla \boldsymbol \zeta)}{\Lambda^{s_2}\boldsymbol
\zeta}} \\
&\hspace{150pt} \leq C\aiminnorm{\Lambda^{s_2}(\boldsymbol u_1 \cdot \nabla \boldsymbol \zeta)-\boldsymbol u_1 \cdot (\Lambda^{s_2}\nabla \boldsymbol \zeta)}_{L^2}
\aiminnorm{\Lambda^{s_2}\boldsymbol \zeta}_{L^2}.
\end{split}
\end{equation}
Applying Lemma~\ref{lem2.5} for
\begin{equation} \label{eq4.2.20.0}
p_1, p_2, q_1, q_2 >2  \hspace{10pt} \text{and} \hspace{10pt}
\frac{1}{p_1}+\frac{1}{p_2}= \frac{1}{q_1}+\frac{1}{q_2}=\frac{1}{2},
\end{equation}
we have
\begin{equation}\label{eq4.2.20.2}
\begin{split}
\aiminnorm{\Lambda^{s_2}(\boldsymbol u_1 \cdot \nabla \boldsymbol \zeta)-\boldsymbol u_1 \cdot (\Lambda^{s_2}\nabla \boldsymbol \zeta)}_{L^2}
& \leq C(\aiminnorm{\nabla \boldsymbol u_1}_{L^{p_1}}\aiminnorm{\Lambda^{s_2} \boldsymbol \zeta}_{L^{p_2}}+\aiminnorm{\Lambda^{s_2}\boldsymbol u_1}_{L^{q_1}}\aiminnorm{\nabla \boldsymbol \zeta}_{L^{q_2}}) \\
& \leq C(\aiminnorm{\Lambda \boldsymbol u_1}_{L^{p_1}}\aiminnorm{\Lambda^{s_2} \boldsymbol \zeta}_{L^{p_2}}+\aiminnorm{\Lambda^{s_2}\boldsymbol u_1}_{L^{q_1}}\aiminnorm{\Lambda \boldsymbol \zeta}_{L^{q_2}}).
\end{split}
\end{equation}
Let
$$ p_1=\frac{2}{\alpha}, \qquad p_2=\frac{2}{1-\alpha}, \qquad q_1=\frac{2}{1-\alpha}, \qquad q_2=\frac{2}{\alpha}.$$ Then, using the Sobolev embedding inequalities, we obtain
\[
\aiminnorm{\Lambda \boldsymbol u_1}_{L^{p_1}}
\leq C\aiminnorm{\Lambda^{2-\alpha} \boldsymbol u_1}_{L^2}
\leq C\aiminnorm{\Lambda^{s_2+\alpha} \boldsymbol u_1}_{L^2},
\]

\[
\aiminnorm{\Lambda^{s_2} \boldsymbol \zeta }_{L^{p_2}}
\leq C\aiminnorm{\Lambda^{s_2+\alpha} \boldsymbol \zeta }_{L^2},
\]

\[
\aiminnorm{\Lambda^{s_2} \boldsymbol u_1 }_{L^{q_1}}
\leq C\aiminnorm{\Lambda^{s_2+\alpha} \boldsymbol u_1}_{L^2},
\]
and
\[
\aiminnorm{\Lambda \boldsymbol \zeta}_{L^{q_2}}
\leq C\aiminnorm{\Lambda^{2-\alpha} \boldsymbol \zeta }_{L^2}
\leq C\aiminnorm{\Lambda^{s_2+\alpha} \boldsymbol \zeta }_{L^2}.
\]
Then,
\begin{equation}\label{eq4.2.20.3}
\aiminnorm{\Lambda^{s_2}(\boldsymbol u_1 \cdot \nabla \boldsymbol \zeta)-\boldsymbol u_1 \cdot (\Lambda^{s_2}\nabla \boldsymbol \zeta)}_{L^2}
\leq C\aiminnorm{\Lambda^{s_2+\alpha} \boldsymbol u_1}_{L^2}\aiminnorm{\Lambda^{s_2+\alpha} \boldsymbol \zeta }_{L^2}.
\end{equation}
Hence, using \eqref{eq4.2.20.6}-\eqref{eq4.2.20.1} and Young's inequality, we have
\begin{equation}\label{eq4.2.20.4}
\begin{split}
\aiminabs{\aimininner{\boldsymbol u_1 \cdot \nabla \zeta}{\Lambda^{2s_2}\zeta}}
&\leq C\aiminnorm{\Lambda^{s_2+\alpha} \boldsymbol u_1}_{L^2}\aiminnorm{\Lambda^{s_2+\alpha} \boldsymbol \zeta }_{L^2}\aiminnorm{\Lambda^{s_2}\boldsymbol \zeta}_{L^2} \\
&\leq \frac{C}{\nu}\aiminnorm{\Lambda^{s_2+\alpha} \boldsymbol u_1}_{L^2}^2\aiminnorm{\Lambda^{s_2}\boldsymbol \zeta}_{L^2}^2+ \frac{\nu}{6}\aiminnorm{\Lambda^{s_2+\alpha} \boldsymbol \zeta }_{L^2}^2.
\end{split}
\end{equation}
By the Cauchy-Schwarz inequality, we have
\begin{equation}
\begin{split}
\aiminabs{\aimininner{\boldsymbol \zeta \cdot \nabla \boldsymbol u_2}{\Lambda^{2s_2}\boldsymbol \zeta}} &= \aiminabs{\aimininner{\Lambda^{s_2-\alpha}(\boldsymbol \zeta \cdot \nabla \boldsymbol u_2)}{\Lambda^{s_2+\alpha}\boldsymbol \zeta}}\\
&\leq \frac{C}{\nu}\aiminnorm{\Lambda^{s_2-\alpha}(\boldsymbol \zeta \cdot \nabla \boldsymbol u_2)}_{L^2}^2+\frac{\nu}{6} \aiminnorm{\Lambda^{s_2+\alpha}\boldsymbol \zeta}_{L^2}^2.
\end{split}
\end{equation}
Again we apply Lemma~\ref{lem2.4} for $p_1$, $p_2$, $q_1$, $q_2$ satisfying \eqref{eq4.2.20.0} to obtain
\begin{equation}
\aiminnorm{\Lambda^{s_2-\alpha}(\boldsymbol \zeta \cdot \nabla \boldsymbol u_2)}_{L^2}
\leq C(\aiminnorm{\Lambda^{s_2-\alpha} \boldsymbol \zeta }_{L^{p_1}}\aiminnorm{\Lambda \boldsymbol u_2}_{L^{p_2}}+\aiminnorm{\boldsymbol \zeta}_{L^{q_1}}\aiminnorm{\Lambda^{s_2-\alpha+1}\boldsymbol u_2}_{L^{q_2}}).
\end{equation}
Let $$p_1=\frac{2}{1-\alpha}, \qquad p_2=\frac{2}{\alpha}, \qquad q_1=\frac{2}{2\alpha-1}, \qquad q_2=\frac{1}{1-\alpha}.$$ Then, using the Sobolev embedding inequalities, we have
\[
\aiminnorm{\Lambda^{s_2-\alpha} \boldsymbol \zeta }_{L^{p_1}}
\leq C\aiminnorm{\Lambda^{s_2} \boldsymbol \zeta }_{L^2},
\]

\[
\aiminnorm{\Lambda \boldsymbol u_2 }_{L^{p_2}}
\leq C\aiminnorm{\Lambda^{2-\alpha} \boldsymbol u_2 }_{L^2}
\leq C\aiminnorm{\Lambda^{s_2+\alpha} \boldsymbol u_2 }_{L^2},
\]

\[
\aiminnorm{\boldsymbol \zeta}_{L^{q_1}}
\leq C\aiminnorm{\Lambda^{2-2\alpha} \boldsymbol \zeta}_{L^2}
\leq C\aiminnorm{\Lambda^{s_2} \boldsymbol \zeta}_{L^2},
\]
and
\[
\aiminnorm{\Lambda^{s_2-\alpha+1} \boldsymbol u_2}_{L^{q_2}}
\leq C\aiminnorm{\Lambda^{s_2+\alpha} \boldsymbol u_2}_{L^2}.
\]
Then, we have
\begin{equation}\label{eq4.2.22}
\aiminnorm{\Lambda^{s_2-\alpha}(\boldsymbol \zeta \cdot \nabla \boldsymbol u_2)}_{L^2}^2
\leq C\aiminnorm{\Lambda^{s_2} \boldsymbol \zeta }_{L^2}^2\aiminnorm{\Lambda^{s_2+\alpha} \boldsymbol u_2 }_{L^2}^2,
\end{equation}
and hence,
\begin{equation}\label{eq4.2.25}
\aiminabs{\aimininner{\boldsymbol \zeta \cdot \nabla \boldsymbol u_2}{\Lambda^{2s_2}\boldsymbol \zeta}} \leq \frac{C}{\nu}\aiminnorm{\Lambda^{s_2} \boldsymbol \zeta }_{L^2}^2\aiminnorm{\Lambda^{s_2+\alpha} \boldsymbol u_2 }_{L^2}^2+\frac{\nu}{6} \aiminnorm{\Lambda^{s_2+\alpha}\boldsymbol \zeta}_{L^2}^2.
\end{equation}
Next, we estimate the term $\aimininner{\boldsymbol u_1 \cdot \nabla \eta}{\Lambda^{2s_1}\eta}$. Since $\aimininner{\boldsymbol u_1 \cdot \nabla (\Lambda^{s_1}\eta)}{\Lambda^{s_1}\eta}=0$, we have\begin{equation} \label{eq4.2.25.1}
\aimininner{\boldsymbol u_1 \cdot \nabla \eta}{\Lambda^{2s_1}\eta}
= \aimininner{\Lambda^{s_1}(\boldsymbol u_1 \cdot \nabla \eta)}{\Lambda^{s_1}\eta}
= \aimininner{\Lambda^{s_1}(\boldsymbol u_1 \cdot \nabla \eta)-\boldsymbol u_1 \cdot \nabla (\Lambda^{s_1}\eta)}{\Lambda^{s_1}\eta}.
\end{equation}
Again, since $\nabla$ and $\Lambda$ commute, we have
\begin{equation}\label{eq4.2.24.1}
\begin{split}
&\aiminabs{\aimininner{\Lambda^{s_1}(\boldsymbol u_1 \cdot \nabla \eta)-\boldsymbol u_1 \cdot \nabla (\Lambda^{s_1}\eta)}{\Lambda^{s_1}\eta}}
=\aiminabs{\aimininner{\Lambda^{s_1}(\boldsymbol u_1 \cdot \nabla \eta)-\boldsymbol u_1 \cdot (\Lambda^{s_1}\nabla \eta)}{\Lambda^{s_1}\eta}} \\
&\hspace{150pt} \leq C\aiminnorm{\Lambda^{s_1}(\boldsymbol u_1 \cdot \nabla \eta)-\boldsymbol u_1 \cdot (\Lambda^{s_1}\nabla \eta)}_{L^2}
\aiminnorm{\Lambda^{s_1}\eta}_{L^2}.
\end{split}
\end{equation}
Applying Lemma~\ref{lem2.5} for $p_1$, $p_2$, $q_1$, $q_2$ satisfying \eqref{eq4.2.20.0},
we have
\begin{equation}\label{eq4.2.24.2}
\begin{split}
\aiminnorm{\Lambda^{s_1}(\boldsymbol u_1 \cdot \nabla \eta)-\boldsymbol u_1 \cdot (\Lambda^{s_1}\nabla \eta)}_{L^2}
& \leq C(\aiminnorm{\nabla \boldsymbol u_1}_{L^{p_1}}\aiminnorm{\Lambda^{s_1} \eta}_{L^{p_2}}+\aiminnorm{\Lambda^{s_1}\boldsymbol u_1}_{L^{q_1}}\aiminnorm{\nabla \eta}_{L^{q_2}}) \\
& \leq C(\aiminnorm{\Lambda \boldsymbol u_1}_{L^{p_1}}\aiminnorm{\Lambda^{s_1} \eta}_{L^{p_2}}+\aiminnorm{\Lambda^{s_1}\boldsymbol u_1}_{L^{q_1}}\aiminnorm{\Lambda \eta}_{L^{q_2}}).
\end{split}
\end{equation}
We set $p_1=2/\beta$ and $p_2=2/(1-\beta)$. Then, using the Sobolev embedding inequalities, we obtain
\[
\aiminnorm{\Lambda \boldsymbol u_1}_{L^{p_1}}
\leq C\aiminnorm{\Lambda^{2-\beta} \boldsymbol u_1}_{L^2}
\leq C\aiminnorm{\Lambda^{s_2+\alpha} \boldsymbol u_1}_{L^2},
\]
and
\[
\aiminnorm{\Lambda^{s_1} \eta }_{L^{p_2}}
\leq C\aiminnorm{\Lambda^{s_1+\beta} \eta }_{L^2}.
\]
Let $\gamma = \mathrm{min}\{\alpha, \beta \}$, $q_1=2/(1-\gamma)$ and $q_2=2/\gamma$. Since $s_1 > 2\mathrm{max} \{1-\alpha, 1-\beta\}$, then $s_1 > 2(1-\gamma)$. Thus,
we have $2-\gamma \leq s_1+\gamma$, and
\[
\aiminnorm{\Lambda^{s_1} \boldsymbol u_1 }_{L^{q_1}}
\leq C\aiminnorm{\Lambda^{s_1+\gamma} \boldsymbol u_1}_{L^2}
\leq C\aiminnorm{\Lambda^{s_1+\alpha} \boldsymbol u_1}_{L^2}
\leq C\aiminnorm{\Lambda^{s_2+\alpha} \boldsymbol u_1}_{L^2},
\]

\[
\aiminnorm{\Lambda \eta}_{L^{q_2}}
\leq C\aiminnorm{\Lambda^{2-\gamma} \eta }_{L^2}
\leq C\aiminnorm{\Lambda^{s_1+\gamma} \eta }_{L^2}
\leq C\aiminnorm{\Lambda^{s_1+\beta} \eta }_{L^2}.
\]
Hence,
\begin{equation}\label{eq4.2.20.3}
\aiminnorm{\Lambda^{s_1}(\boldsymbol u_1 \cdot \nabla \eta)-\boldsymbol u_1 \cdot (\Lambda^{s_1}\nabla \eta)}_{L^2}
\leq C\aiminnorm{\Lambda^{s_2+\alpha} \boldsymbol u_1}_{L^2}\aiminnorm{\Lambda^{s_1+\beta} \eta }_{L^2}.
\end{equation}
Therefore, using \eqref{eq4.2.25.1}-\eqref{eq4.2.24.1} and Young's inequality, we have
\begin{equation}\label{eq4.2.20.5}
\begin{split}
\aiminabs{\aimininner{\boldsymbol u_1 \cdot \nabla \zeta}{\Lambda^{2s_1}\eta}}
&\leq C\aiminnorm{\Lambda^{s_2+\alpha} \boldsymbol u_1}_{L^2}\aiminnorm{\Lambda^{s_1+\beta} \boldsymbol \zeta }_{L^2}\aiminnorm{\Lambda^{s_1}\eta}_{L^2} \\
&\leq \frac{C}{\kappa}\aiminnorm{\Lambda^{s_2+\alpha} \boldsymbol u_1}_{L^2}^2\aiminnorm{\Lambda^{s_1}\eta}_{L^2}^2+ \frac{\kappa}{6}\aiminnorm{\Lambda^{s_1+\beta} \eta }_{L^2}^2.
\end{split}
\end{equation}
For the term $\aimininner{\boldsymbol \zeta \cdot \nabla \theta_2}{\Lambda^{2s_1}\eta}$, we consider two cases: $s_1 \geq \beta$ and $s_1 < \beta$. Suppose that $s_1 \geq \beta$, so that
\begin{equation}
\aiminabs{\aimininner{\boldsymbol \zeta \cdot \nabla \theta_2}{\Lambda^{2s_1}\eta}} = \aiminabs{\aimininner{\Lambda^{s_1-\beta}(\boldsymbol \zeta \cdot \nabla \theta_2)}{\Lambda^{s_1+\beta}\eta}}
\leq \frac{C}{\kappa}\aiminnorm{\Lambda^{s_1-\beta}(\boldsymbol \zeta \cdot \nabla \theta_2)}_{L^2}^2+\frac{\kappa}{6} \aiminnorm{\Lambda^{s_1+\beta}\eta}_{L^2}^2.
\end{equation}
Applying Lemma~\ref{lem2.4} for $p_1$, $p_2$, $q_1$, $q_2$ satisfying \eqref{eq4.2.20.0}, we have
\begin{equation}
\aiminnorm{\Lambda^{s_1-\beta}(\boldsymbol \zeta \cdot \nabla \theta_2)}_{L^2}
\leq C(\aiminnorm{\Lambda^{s_1-\beta} \boldsymbol \zeta }_{L^{p_1}}\aiminnorm{\Lambda \theta_2}_{L^{p_2}}+\aiminnorm{\boldsymbol \zeta}_{L^{q_1}}\aiminnorm{\Lambda^{s_1-\beta+1}\theta_2}_{L^{q_2}}).
\end{equation}
Let $$p_1=\frac{2}{1-\beta}, \qquad p_2=\frac{2}{\beta}, \qquad q_1=\frac{2}{2\beta-1}, \qquad q_2=\frac{1}{1-\beta}.$$ Then, using the Sobolev embedding inequalities, we obtain
\[
\aiminnorm{\Lambda^{s_1-\beta} \boldsymbol \zeta }_{L^{p_1}}
\leq C\aiminnorm{\Lambda^{s_1} \boldsymbol \zeta }_{L^2}
\leq C\aiminnorm{\Lambda^{s_2} \boldsymbol \zeta }_{L^2},
\]

\[
\aiminnorm{\Lambda \theta_2 }_{L^{p_2}}
\leq C\aiminnorm{\Lambda^{2-\beta} \theta_2 }_{L^2}
\leq C\aiminnorm{\Lambda^{s_1+\beta} \theta_2 }_{L^2},
\]

\[
\aiminnorm{\boldsymbol \zeta}_{L^{q_1}}
\leq C\aiminnorm{\Lambda^{2-2\beta} \boldsymbol \zeta}_{L^2}
\leq C\aiminnorm{\Lambda^{s_2} \boldsymbol \zeta}_{L^2},
\]
and
\[
\aiminnorm{\Lambda^{s_1-\beta+1} \theta_2 }_{L^{q_2}}
\leq C\aiminnorm{\Lambda^{s_1+\beta} \theta_2 }_{L^2}.
\]
Then, we have
\begin{equation}\label{eq4.2.22}
\aiminnorm{\Lambda^{s_1-\beta}(\boldsymbol \zeta \cdot \nabla \theta_2 )}_{L^2}^2
\leq C\aiminnorm{\Lambda^{s_2} \boldsymbol \zeta }_{L^2}^2\aiminnorm{\Lambda^{s_1+\beta} \theta_2 }_{L^2}^2.
\end{equation}
Hence, using Young's inequality, we have
\begin{equation}\label{eq4.25}
\aiminabs{\aimininner{\boldsymbol \zeta \cdot \nabla \theta_2 }{\Lambda^{2s_1}\eta}} \leq \frac{C}{\kappa}\aiminnorm{\Lambda^{s_2} \boldsymbol \zeta }_{L^2}^2\aiminnorm{\Lambda^{s_1+\beta} \theta_2}_{L^2}^2+\frac{\kappa}{6} \aiminnorm{\Lambda^{s_1+\beta}\eta}_{L^2}^2.
\end{equation}
Otherwise, $s_1 < \beta$ so that $2s_1 < s_1+\beta$.
Using the Cauchy-Schwarz inequality, Poincar\'e's inequality and Young's inequality, we have
\begin{equation}
\begin{split}
\aiminabs{\aimininner{\boldsymbol \zeta \cdot \nabla \theta_2}{\Lambda^{2s_1}\eta}}
&\leq \aiminnorm{\boldsymbol \zeta \cdot \nabla \theta_2}_{L^2}\aiminnorm{\Lambda^{2s_1} \eta}_{L^2} \\
& \leq C\aiminnorm{\boldsymbol \zeta \cdot \nabla \theta_2}_{L^2}
\aiminnorm{\Lambda^{s_1+\beta} \eta}_{L^2} \\
& \leq \frac{C}{\kappa}\aiminnorm{\boldsymbol \zeta \cdot \nabla \theta_2}_{L^2}^2+
\frac{\kappa}{6}\aiminnorm{\Lambda^{s_1+\beta} \eta}_{L^2}^2.
\end{split}
\end{equation}
For any $p_1, p_2 >2$ such that $1/p_1+1/{p_2}= 1/2$, we have
\[
\aiminnorm{\boldsymbol \zeta \cdot \nabla \theta_2}_{L^2} \leq \aiminnorm{\boldsymbol \zeta}_{L^{p_1}}\aiminnorm{\nabla \theta_2}_{L^{p_2}} \leq
C\aiminnorm{\boldsymbol \zeta}_{L^{p_1}}\aiminnorm{\Lambda \theta_2}_{L^{p_2}}.
\]
Let $p_1= 2/(1-\beta)$ and $p_2=2/\beta$. Since $\beta \leq 1 \leq s_2$, using the Sobolev embedding inequality, we have
\begin{equation}
\aiminnorm{\boldsymbol \zeta}_{L^{p_1}} \leq C\aiminnorm{\Lambda^{\beta} \boldsymbol \zeta}_{L^2} \leq C\aiminnorm{\Lambda^{s_2} \boldsymbol \zeta}_{L^2},
\end{equation}
and
\begin{equation}
\aiminnorm{\Lambda \theta_2}_{L^{p_2}} \leq C\aiminnorm{\Lambda^{2-\beta} \theta_2}_{L^2} \leq C\aiminnorm{\Lambda^{s_1+\beta} \theta_2}_{L^2}.
\end{equation}
Hence,
\begin{equation} \label{eq4.26}
\aiminabs{\aimininner{\boldsymbol \zeta \cdot \nabla \theta_2 }{\Lambda^{2s_1}\eta}} \leq \frac{C}{\kappa}\aiminnorm{\Lambda^{s_2} \boldsymbol \zeta }_{L^2}^2\aiminnorm{\Lambda^{s_1+\beta} \theta_2}_{L^2}^2+ \frac{\kappa}{6} \aiminnorm{\Lambda^{s_1+\beta}\eta}_{L^2}^2.
\end{equation}
Therefore, using \eqref{eq4.2.21}, \eqref{eq4.2.20.4}, \eqref{eq4.2.25}, \eqref{eq4.2.20.5}, \eqref{eq4.25} and \eqref{eq4.26}, we have
\begin{equation}\label{eq4.2.26}
\begin{split}
\frac{1}{2} \frac{\mathrm{d}}{\mathrm{d}t} \aiminnorm{\Lambda^{s_2}\boldsymbol \zeta}_{L^2}^2 + \nu \aiminnorm{\Lambda^{s_2+\alpha} \boldsymbol \zeta}_{L^2}^2
\leq &\frac{\kappa}{6} \aiminnorm{\Lambda^{s_1+\beta}\eta}_{L^2}^2
+ \frac{\nu}{2} \aiminnorm{\Lambda^{s_2+\alpha}\boldsymbol \zeta}_{L^2}^2\\
&+C(\aiminnorm{\Lambda^{s_2+\alpha} \boldsymbol u_1}_{L^2}^2+\aiminnorm{\Lambda^{s_2+\alpha} \boldsymbol u_2}_{L^2}^2)\aiminnorm{\Lambda^{s_2}\boldsymbol \zeta}_{L^2}^2,
\end{split}
\end{equation}
and
\begin{equation}\label{eq4.2.27}
\begin{split}
\frac{1}{2} \frac{\mathrm{d}}{\mathrm{d}t} \aiminnorm{\Lambda^{s_1}\eta}_{L^2}^2 + &\kappa \aiminnorm{\Lambda^{s_1+\beta} \eta}_{L^2}^2
\leq \frac{\kappa}{3} \aiminnorm{\Lambda^{s_1+\beta}\theta}_{L^2}^2\\
&\qquad+ C\aiminnorm{\Lambda^{s_2+\alpha} \boldsymbol u_1}_{L^2}^2 \aiminnorm{\Lambda^{s_1}\eta}_{L^2}^2
+C\aiminnorm{\Lambda^{s_1+\beta} \theta_2}_{L^2}^2 \aiminnorm{\Lambda^{s_2}\boldsymbol \zeta}_{L^2}^2.
\end{split}
\end{equation}
Summing \eqref{eq4.2.26} and \eqref{eq4.2.27} gives
\[
\begin{split}
\frac{\mathrm{d}}{\mathrm{d}t} &(\aiminnorm{\Lambda^{s_2}\boldsymbol \zeta}_{L^2}^2+\aiminnorm{\Lambda^{s_1}\eta}_{L^2}^2) +
\nu \aiminnorm{\Lambda^{s_2+\alpha} \boldsymbol \zeta}_{L^2}^2 +\kappa \aiminnorm{\Lambda^{s_1+\beta} \eta}_{L^2}^2 \\
&\leq C(\aiminnorm{\Lambda^{s_2+\alpha} \boldsymbol u_1}_{L^2}^2+\aiminnorm{\Lambda^{s_2+\alpha} \boldsymbol u_2}_{L^2}^2+\aiminnorm{\Lambda^{s_2+\beta} \theta_2}_{L^2}^2 )\aiminnorm{\Lambda^{s_2}\boldsymbol \zeta}_{L^2}^2+C\aiminnorm{\Lambda^{s_2+\alpha} \boldsymbol u_1}_{L^2}^2 \aiminnorm{\Lambda^{s_1}\eta}_{L^2}^2 \\
&\leq C(\aiminnorm{\Lambda^{s_2+\alpha} \boldsymbol u_1}_{L^2}^2+\aiminnorm{\Lambda^{s_2+\alpha} \boldsymbol u_2}_{L^2}^2+\aiminnorm{\Lambda^{s_1+\beta} \theta_2}_{L^2}^2 )
(\aiminnorm{\Lambda^{s_2}\boldsymbol \zeta}_{L^2}^2+\aiminnorm{\Lambda^{s_1}\eta}_{L^2}^2),
\end{split}
\]
and using the Gronwall inequality in Lemma~\ref{lemma2.0.1}, we obtain that
\[
\begin{split}
&\aiminnorm{\Lambda^{s_2}\boldsymbol \zeta(t)}_{L^2}^2+\aiminnorm{\Lambda^{s_1}\eta(t)}_{L^2}^2 \\
\leq &C(\aiminnorm{\Lambda^{s_2}\boldsymbol \zeta(0)}_{L^2}^2+\aiminnorm{\Lambda^{s_1}\eta(0)}_{L^2}^2)
\mathrm{exp} \left\{\int_0^t \aiminnorm{\Lambda^{s_2+\alpha} \boldsymbol u_1(s)}_{L^2}^2+\aiminnorm{\Lambda^{s_2+\alpha} \boldsymbol u_2(s)}_{L^2}^2+\aiminnorm{\Lambda^{s_1+\beta} \theta_2(s)}_{L^2}^2 {\mathrm{d}s}    \right\}.
\end{split}
\]
Notice that from \eqref{eq3.5.11} that
\[
\int_0^T \aiminnorm{\Lambda^{s_2+\alpha} \boldsymbol u_1(s)}_{L^2}^2 {\mathrm{d}s} < \infty,\hspace{3pt} \int_0^T \aiminnorm{\Lambda^{s_2+\alpha} \boldsymbol u_2(s)}_{L^2}^2 {\mathrm{d}s} < \infty, \hspace{3pt} \mathrm{and} \hspace{2pt}\int_0^T \aiminnorm{\Lambda^{s_1+\beta} \theta_2(s)}_{L^2}^2 {\mathrm{d}s} < \infty.
\]
By the Riesz lemma, since $\aiminnorm{\Lambda^{s_1}\eta(0)}_{L^2}$ and $\aiminnorm{\Lambda^{s_2}\boldsymbol \zeta(0)}_{L^2}$ go to zero, then $\aiminnorm{\Lambda^{s_1}\eta(t)}_{L^2}$ and $\aiminnorm{\Lambda^{s_2}\boldsymbol \zeta(t)}_{L^2}$ converge to zero for almost every $t$. Since $\aiminnorm{\Lambda^{s_1}\eta}_{L^2}$ and $\aiminnorm{\Lambda^{s_2}\boldsymbol \zeta}_{L^2}$ are continuous in $t$ as we have proved in Section \ref{sec4.1}, then $(\aiminnorm{\Lambda^{s_1}\eta(0)}_{L^2}$, $\aiminnorm{\Lambda^{s_2}\boldsymbol \zeta(0)}_{L^2})$ converges to zero for every $t$.
\end{proof}

\section*{Acknowledgments}
The authors would like to thank Prof. M. Jolly for an interesting discussion on this work.
AH would like to thank Prof. V. Pata for the discussion on the variants of the Gronwall lemma.
This work was partially supported by the National Science Foundation under the grant NSF DMS-1418911.3, and by the Research Fund of Indiana University.

\bibliographystyle{amsalpha}
\providecommand{\bysame}{\leavevmode\hbox to3em{\hrulefill}\thinspace}
\providecommand{\MR}{\relax\ifhmode\unskip\space\fi MR }
\providecommand{\MRhref}[2]{%
  \href{http://www.ams.org/mathscinet-getitem?mr=#1}{#2}
}
\providecommand{\href}[2]{#2}

\end{document}